\documentclass[11pt]{amsart}

\usepackage{amsmath,amsfonts,amsthm,amssymb,
amsthm,graphicx,mathtools,amscd, enumerate, paralist, bm}
\usepackage{hyperref}

\usepackage[mathscr]{eucal}
\usepackage{graphics, color, epsfig}

\theoremstyle{plain}
\newtheorem{theorem}{Theorem}[section]
\newtheorem{lemma}[theorem]{Lemma}
\newtheorem{corollary}[theorem]{Corollary}
\newtheorem{definition}[theorem]{Definition}

\newtheorem{proposition}[theorem]{Proposition}

\newcommand{\la}{\lambda}
\newcommand{\eps}{\varepsilon}

\newcommand{\al}{\alpha}

\newcommand{\gam}{\gamma}
\newcommand{\kap}{\kappa}

\newcommand{\sig}{\sigma}
\newcommand{\del}{\delta}
\newcommand{\om}{\omega}
\newcommand{\Gam}{\mathnormal{\Gamma}}

\newcommand{\Th}{\mathnormal{\Theta}}
\newcommand{\La}{\mathnormal{\Lambda}}

\newcommand{\Sig}{\mathnormal{\Sigma}}

\newcommand{\Ph}{\mathnormal{\Phi}}
\newcommand{\Ps}{\mathnormal{\Psi}}
\newcommand{\Om}{\mathnormal{\Omega}}

\newcommand{\N}{{\mathbb N}}

\newcommand{\R}{{\mathbb R}}

\newcommand{\Z}{{\mathbb Z}}
\newcommand{\Sb}{{\mathbb S}}

\newcommand{\BB}{{\mathbb B}}
\newcommand{\EE}{{\mathbb E}}
\newcommand{\E}{{\mathbb E}}

\newcommand{\PP}{{\mathbb P}}
\newcommand{\NN}{{\mathbb N}}
\newcommand{\RR}{{\mathbb R}}

\newcommand{\ONE}{\boldsymbol{1}}

\newcommand{\calA}{{\mathcal A}}
\newcommand{\calB}{{\mathcal B}}

\newcommand{\calD}{{\mathcal D}}

\newcommand{\calF}{{\mathcal F}}
\newcommand{\calG}{{\mathcal G}}

\newcommand{\calL}{{\mathcal L}}

\newcommand{\calO}{{\mathcal O}}

\newcommand{\calS}{{\mathcal S}}

\newcommand{\bD}{{\mathbf D}}

\newcommand{\bX}{{\mathbf X}}

\newcommand{\frS}{\mathfrak{S}}

\renewcommand{\proof}{\noindent{\bf Proof.\ }}

\newcommand{\diag}{{\rm diag}}

\newcommand{\w}{\wedge}

\newcommand{\pl}{\partial}

\newcommand{\To}{\Rightarrow}

\newcommand{\iy}{\infty}

\newcommand{\cadlag}{c\`adl\`ag }
\newcommand{\noi}{\noindent}

\newcommand{\bi}{\begin{itemize}}
\newcommand{\ei}{\end{itemize}}
\newcommand{\veps}{\varepsilon}
\newcommand{\ti}{\tilde}

\newcommand{\st}{\eta}
\newcommand{\tht}{\zeta}

\newcommand{\one}{\mathbf{1}}
\newcommand{\PS}{{\rm PS}}
\newcommand{\CP}{{\rm CP}}
\newcommand{\es}{\tiny{\emptyset}}
\newcommand{\osc}{{\rm osc}}
\newcommand{\Id}{{\rm Id}}

\usepackage{tcolorbox}
\newtcolorbox{bx}{colback=white!5!white,colframe=blue!75!black}

\begin{document}

\title[Non-SRBM Diffusion Limits]{Diffusion Limits in the Quarter Plane and Non-Semimartingale Reflected Brownian Motion
}

\date{February 29, 2024}
\author{Rami Atar}
\address{Viterbi Faculty of Electrical and Computer Engineering
\\
Technion, Haifa 3200003, Israel
}
\email{rami@technion.ac.il}

\author{Amarjit Budhiraja}
\address{Department of Statistics and Operations Research
\\
University of North Carolina at Chapel Hill\\
North Carolina 27599\\
USA
}
\email{budhiraj@email.unc.edu}


\begin{abstract}
We consider a continuous-time random walk
in the quarter plane for which the transition intensities
are constant on each of the four faces $(0,\iy)^2$, $F_1=\{0\}\times(0,\iy)$,
$F_2=(0,\iy)\times\{0\}$ and $\{(0,0)\}$. We show that when rescaled diffusively
it converges in law to a Brownian motion with oblique reflection direction $d^{(i)}$
on face $F_i$, $i=1,2$,
defined via the Varadhan-Williams submartingale problem \cite{vw}.
A parameter denoted by $\al$ was introduced in \cite{vw}, measuring the extent
to which $d^{(i)}$ are inclined toward the origin. In the case of the quarter plane,
$\al$ takes values in $(-2,2)$,
and it is known that the reflected Brownian motion is a semimartingale if and only if
$\al\in(-2,1)$. Convergence results via both the Skorohod map and the invariance principle for semimartingale reflected Brownian motion are known
to hold in various settings in arbitrary dimension.  In the case of the quarter plane, the invariance principle was proved for $\alpha \in (-2,1)$ whereas for tools based on the Skorohod map to be applicable it is necessary (but not sufficient) that $\alpha \in [-1,1)$.
Another tool that has been used to prove convergence
in general dimension is the extended Skorohod map, which in the case of the quarter plane
provides convergence for $\alpha=1$.
This paper focuses on the range $\alpha \in (1,2)$, where the Skorohod problem
and the extended Skorohod problem do not possess a unique solution, the limit process is not a semimartingale, 
and convergence to reflected Brownian motion has not been shown before.
The result has implications on the asymptotic analysis of two Markovian
queueing models:
The {\it generalized processor sharing model with parallelization slowdown},
and the {\it coupled processor model}.
In both cases, the diffusion limit in heavy traffic is characterized by
the aforementioned reflected Brownian motion.
The restriction of our treatment to dimension 2
is due to the fact that, for analogous models in higher dimension, the well posedness of the submartingale
problem for the candidate limit process is an open problem.
\end{abstract}

\maketitle

\section{Introduction}\label{sec1}
This paper studies a Markov process on $\calS^1\doteq\Z_+^2$ that,
under diffusion scaling, converges to a reflected Brownian motion
(RBM) in $\calS\doteq[0,\iy)^2$ with oblique reflection on the boundary
$\pl\calS$. The focus is on the case where
the directions of reflection, that are constant on each of
the two faces $F_1=\{0\}\times(0,\iy)$ and $F_2=(0,\iy)\times\{0\}$,
are in a range in which the RBM is not a semimartingale,
and has not been constructed
as a pathwise transformation of planar Brownian motion (BM)
but rather via a submartingale problem introduced in \cite{vw}
(where a general wedge was considered).
The latter is defined along the lines of
the Stroock-Varadhan submartingale problem \cite{SV-B} with an additional condition
to account for the behavior of the process at the origin, where
the reflection vector field is discontinuous.
The existence and uniqueness of solutions was established
in \cite{vw} in the case of zero drift,
and extended to a constant drift in \cite{lakner2023reflected}.
The Skorohod map \cite{dupish1} and the extended Skorohod map \cite{ramanan2006reflected},
which act in path space to transform a BM to an RBM,
have been used in conjunction with the continuous mapping theorem
to prove convergence to RBM in a wide variety of settings and in arbitrary dimension.
In the setting under consideration, no pathwise transformation is known,
and, in particular, these tools
are not available. In \cite{williams1998invariance}, tools
that go beyond continuous mapping techniques were introduced, providing
invariance principles for processes that converge to
a semimartingale RBM (SRBM) in the orthant in any dimension.
These do not apply here either because the limit process is not a semimartingale.
Whereas the question of convergence in absence of these tools
provides the main motivation for this work, further interest stems
from the fact that
the setting also describes two Markovian queueing models. These are
the generalized processor sharing (GPS) with parallelization slowdown, and the coupled processor model. The limit theorem we prove establishes
the heavy traffic limit for both models, characterized via the submartingale problem.

To introduce our setting more precisely,
let $\{e^{(1)},e^{(2)}\}$ denote the standard basis in $\R^2$
(considered as column vectors)
and assign an edge between $x,y\in\calS^1$ whenever $x-y=\pm e^{(i)}$
for $i=1$ or $2$, a relation denoted as $x\sim y$. For $n\in\N$, let $X^n$
be a nearest neighbor
Markov process on the graph $(\calS^1,\sim)$ with initial state $X^n(0)=x^n\in\calS^1$ and transition intensities
\begin{equation}\label{e1}
r^n(x,y)=\begin{cases}
\la^n_i & \text{ if } y=x+e^{(i)},\,i=1,2,\\
\mu^n_i & \text{ if } x\in\calS^o,\, y=x-e^{(i)},\,i=1,2,\\
\mu^n_i+\nu^n_i & \text{ if } x\in F_{i^\#},\, y=x-e^{(i)},\, i=1,2,
\end{cases}
\end{equation}
where throughout, $i^\#=3-i$ and $\calS^o$ denotes the interior of $\calS$. Here $\la^n_i, \mu^n_i, \nu^n_i$ are positive scalars for all $n\in \NN$ and $i=1,2$.
Thus, when one of the components is zero, the intensity of downward jumps of the other
component increases. Assume that
\begin{align}
\label{e7-}
\hat\la^n_i &\doteq n^{-1/2}(\la^n_i-n\la_i)\to\hat\la_i,\\
\hat\mu^n_i &\doteq n^{-1/2}(\mu^n_i-n\mu_i)\to\hat\mu_i,
\label{e7}
\\
\label{e7+}
\hat\nu^n_i &\doteq n^{-1/2}(\nu^n_i-n\nu_i)\to\hat\nu_i,
\end{align}
for $i=1,2$, as $n\to\iy$,
where $\la_i,\mu_i,\nu_i\in(0,\iy)$ and $\hat\la_i,\hat\mu_i,\hat\nu_i\in\R$. Assume moreover that $\la_i=\mu_i$, $i=1,2$;
in the queueing literature this is referred to as the {\em heavy traffic}
condition.
Let $\hat X^n=n^{-1/2}X^n$ and assume that the rescaled initial conditions
converge, namely $\hat x^n\doteq n^{-1/2}x^n\to\hat x\in\calS$.
As long as $\hat X^n$ visits only sites in $\calS^o\cup F_i$,
for either $i=1$ or $2$,
it is well approximated by a BM with oblique reflection on this face.
More precisely, let $\calO$ be a smooth open planar domain
with $\hat x\in\calO \cap \calS$. If, for either $i=1$ or $2$, 
$\bar \calO \cap \{x \in S: x_{i^\#}=0\} = \emptyset$,
then the sequence of processes $\hat X^n$ stopped on exiting $\calO$,
converges in law to an RBM on the half space $\{x \in \RR^2: x_i\ge 0\}$ stopped on
exiting $\calO$. This RBM has drift and diffusion coefficients
\begin{equation}\label{095}
b=(b_1,b_2)', \qquad \Sig=\diag(\sig_1,\sig_2),
\qquad
b_j=\hat\la_j-\hat\mu_j,
\qquad
\sig_j^2=\la_j+\mu_j=2\mu_j,\qquad j=1,2.
\end{equation}
Its reflection vector field on the boundary $\{x \in \RR^2: x_i= 0\}$  takes the constant value
\begin{equation}\label{eq:di}
d^{(i)}=e^{(i)}-\frac{\nu_{i^\#}}{\mu_i}e^{(i^\#)}.
\end{equation}
It is therefore natural to guess that the sequence $\hat X^n$ itself converges in law
to an RBM in $\calS$, with drift and diffusion coefficients as above,
and reflection vector field that, for each $i=1,2$, takes the value $d^{(i)}$ on $F_i$, which is a stochastic
process rigorously defined through the submartingale problem \cite{vw,lakner2023reflected}.
The goal of this work is to prove this convergence for a range of reflection directions $d^{(i)}$ not covered by existing results.

The geometry of the problem is captured in Fig.~\ref{fig1}, where
$\theta^{(i)}$ denotes the angle between the inward normal to face
$F_i$, i.e., $e^{(i)}$, and $d^{(i)}$, considered positive when $d^{(i)}$
inclines toward the origin, namely
\begin{equation}\label{eq:th}
	\theta^{(i)} \doteq \arcsin(-\|d^{(i)}\|^{-1}d^{(i)}_{i^\#}),
	\qquad i=1,2.
\end{equation}

\begin{figure}
\begin{center}
\includegraphics[width=20em]{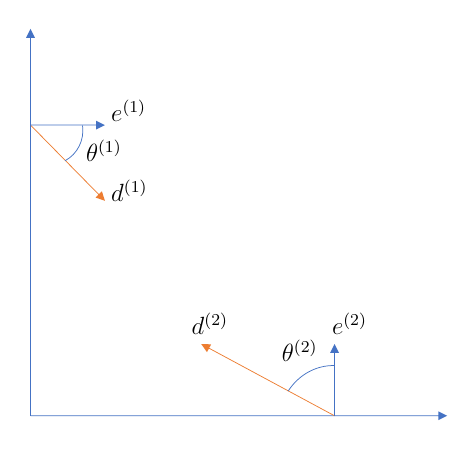}
\end{center}
\vspace{-1em}
\caption{\sl
For $i=1,2$, $d^{(i)}$ is the direction of reflection at face $x_i=0$, and
$\theta^{(i)}$ is the angle between $e^{(i)}$
and $d^{(i)}$, positive when $d^{(i)}$ is toward the origin.
}
\label{fig1}
\end{figure}

To put the goal of this paper in context, consider, in the rest of this
introduction, more general $\Sig$ and $d^{(i)}$ than the ones given
by \eqref{095} and \eqref{eq:di}. Namely, let $\Sig$ be any member
of $\frS$, the set of $2\times2$ diagonal matrices with strictly positive diagonal
entries, and let $d^{(i)}$ merely satisfy the requirement $d^{(i)}_i>0$, $i=1,2$.
When $\Sig=\Id\doteq\diag(1,1)$, we will say that the RBM is
a {\it unit variance} RBM.
The original formulation from \cite{vw} considered a zero drift,
unit variance RBM, with $d^{(i)}$ as above, and proved
the existence and uniqueness in law of the process
defined via the submartingale problem. These results
were recently extended to the case of a general constant drift in \cite{lakner2023reflected}, a result that is used in this paper.
Although the results in \cite{vw,lakner2023reflected},
were formulated for $\Sig=\Id$, they cover
the case of a general $\Sig\in\frS$.
That is, it is intuitively clear and will be made precise
once the formal definition is introduced, that
if $X$ is an RBM with drift, diffusion coefficient, and reflection data $(b, \Sig, \{d^{(i)}\})$, then
$X^* \doteq \Sig^{-1}X$ is an RBM with the data $(b^*, \Id, d^{*,(i)}\})$, where
\begin{equation}\label{094}
b^*=\Sig^{-1} b,\qquad d^{*,(i)}=\Sig^{-1} d^{(i)}, \qquad i=1,2.
\end{equation}
If, analogously to \eqref{eq:th}, we let
\begin{equation}\label{eq:thstar}
	\theta_*^{(i)} \doteq \arcsin(-\|d^{*,(i)}\|^{-1}d^{*,(i)}_{i^\#}),
	\qquad i=1,2,
\end{equation}
then $\theta_*^{(i)}$ is the angle between $e^{(i)}$ and $d^{*,(i)}$,
positive when $d^{*,(i)}$ inclines toward the origin.
A key parameter introduced in \cite{vw} is
\begin{equation}\label{eq:alstar}
\al_*=\frac{\theta_*^{(1)}+\theta_*^{(2)}}{\pi/2}.
\end{equation}
Denoted in \cite{vw} by the symbol $\alpha$, this parameter
measures the degree of ``push'' toward the origin exercised by
the reflection at the boundary, and is known to determine various
properties of the process. In particular, the process is a semimartingale
if and only if $\al_*<1$. Moreover, if $\al_*>0$ then the origin
is a.s.\ visited, and otherwise it is a.s.\ not visited. As it turns out,
the tools available for proving convergence to this process
also depend to a great extent on $\al_*$.
In this work we will be interested in the case $\al_*\in(1,2)$
which is the only range in which convergence has not been shown before.
(In a manner analogous to \eqref{eq:alstar}, one could define
a parameter in terms of the geometry of the original model, namely
$\al \doteq (\pi/2)^{-1}(\theta^{(1)}+\theta^{(2)})$.
It is easy to see that $\al_*\in R$ if and only if $\al\in R$  for each of the ranges
$R=(-2,-1), \{-1\}, (-1,1), \{1\}$ and $(1,2)$.
However, whereas $\al_*$ plays a major analytic role in our proofs,
$\al$ is not as important and will not be used further in this paper.)

Let us briefly review the tools that have been used to prove convergence
when $\al_*\in(-2,1]$.
Results on the aforementioned Skorohod problem for polyhedral domains
\cite{har-rei, dupish1, dup-ram} give sufficient conditions for the transformation
to be Lipschitz continuous in path space, and based on that, the convergence
of diffusively scaled reflected random walks and related queueing models
to an RBM has been shown for a variety of models.
When specialized to the quarter plane, the setting of \cite{har-rei} covers
the case where $\theta^{(i)}\ge0$, $i=1,2$ and $\al_*\in[0,1)$,
and the broader setting of \cite{dupish1,dup-ram} allows for
$\al_*\in[-1,1)$, although the condition
$\al_*\in[-1,1)$ is not sufficient for these results to hold.
In the case $\al_*=1$, the Skorohod map is not well-defined
but the extended Skorohod map introduced in \cite{ramanan2006reflected}
gives a pathwise construction of an RBM and again provides
convergence via continuity properties of this map.

For $\al_*$ in the range $(-2, -1)\cup(1,2)$,
the Skorohod and extended Skorohod problem do not possess unique solutions,
and only weak formulations are available.
In \cite{rei-wil}, a weak formulation of an SRBM in the positive orthant $\RR_+^N$
was given, akin to the notion of a weak solution of a stochastic differential equation.
It was shown that a necessary condition for existence of this process
is the so-called {\it completely-$\calS$} condition, an algebraic
condition on the matrix composed by the directions of reflection $\{d^{(i)}\}$.
In \cite{tay-wil},
existence and uniqueness of this process were established under the completely-$\calS$ condition, and in \cite{williams1998invariance}, under the same condition,
an invariance principle
was proved, yielding convergence to the SRBM.
In the quarter plane, this condition holds if and only if $\al_* \in (-2,1)$,
and in this range the process agrees with the one defined via
the submartingale problem.
The results on the submartingale problem cover the entire range $\al_*\in(-2,2)$,
and thus this is the only formulation that handles
the range $\al_*\in(1,2)$ that is of interest in this paper.

More broadly,
reflected diffusions have been studied in a variety of settings.
In \cite{dai-wil}, existence and uniqueness of an SRBM in convex polyhedral domains with constant direction of reflection on each face were proved. An extension
of the results of \cite{vw} on the submartingale problem for a higher dimensional cone,
with radially homogeneous direction of reflection, was studied in
\cite{kwon-wil}.
Related questions have been studied for domain with cusps. For such domains,
determining whether a reflected diffusion is a semimartingale
has been addressed in \cite{dante1993semimartingale}. More recently,
\cite{cos-kur1} gave existence and uniqueness results in a two-dimensional cusp with varying, oblique directions of reflection, and \cite{cos-kur2} provided such
results for higher dimensional domains with a single singular point.
Conditions under which an obliquely reflected diffusion process
constitutes a Dirichlet process were studied in \cite{kang2010dirichlet}.
When specialized to the case of a wedge, these conditions correspond to
$\al_*=1$. Subsequently, the Dirichlet process property in the case of a wedge with $\al_*\in(1,2)$ was established  in \cite{lakner2019roughness}.
For reflected diffusions in piecewise smooth domains, \cite{kang2017submartingale}
studied the equivalence between well-posedness of the submartingale problem
and  weak existence and uniqueness of solutions for stochastic differential equations with reflection, in settings that
include non-semimartingale reflected diffusions.

There is a rich literature on convergence of discrete processes in domains with boundary
to RBM in dimension $2$ and higher.
As far as queueing models are concerned, the first main result in this direction was
the treatment of the generalized Jackson network in \cite{reiman1984open}, where
convergence of the diffusively scaled queue-length process to a multidimensional RBM
was proved based on continuity of the underlying Skorohod map.
We refer to \cite[Ch.\ 14]{whittbook}
for a survey on convergence results via continuity.

In settings where the Skorohod map does not exist but the RBM is a semimartingale, the aforementioned invariance principle tools of \cite{williams1998invariance} have been used for proving
convergence; see e.g.\ \cite{williams1998diffusion}. The paper \cite{ramanan2003fluid} proves a heavy traffic limit theorem when the RBM is not a semimartingale but given through the extended Skorohod problem \cite{ramanan2006reflected}.
Using different techniques, invariance principles for random walks
converging to RBM in highly non-smooth Euclidean domains
were established in \cite{burdzy2008discrete}. In this setting, the RBM
is defined via the Dirichlet form, and loosely speaking,
the reflection vector field is in the direction normal to the boundary.

A natural desired extension of the setting from \cite{vw} to higher dimension
is to a polyhedral cone with reflection direction that is constant
on each face $F_i$ of codimension $1$, satisfying merely $d^{(i)}\cdot n^{(i)}>0$ for all $i$,
where $n^{(i)}$ denotes the inward normal at face $F_i$. In particular,
the case of an orthant in dimension  $\ge 3$ corresponds to
higher dimensional versions of \eqref{e1} and of the queueing models considered in this paper.
However,
the well-posedness of the corresponding submartingale problem
in this generality remains an open problem.

Finally, we remark that the conjectured connection of diffusion limits
of queueing models to the solution of the submartingale problem
in the quarter plane when $\al_*\in(1,2)$
has previously been proposed in \cite{lakner2019roughness}
and \cite{lakner2023reflected}.

\subsection{Paper organization}
The result and its applications for queueing are presented in Section \ref{sec2},
starting with Section \ref{sec21} which states the definition of the submartingale problem
and the main result. Implications to the heavy traffic
limits of two queueing models are provided in Section \ref{sec22}.
An outline of the proof is described in Section \ref{sec23}.
Several preparatory lemmas are provided in Section \ref{sec3}.
Finally the proof of the main result appears in Section \ref{sec4}.
See Section \ref{sec23} for a detailed description of the content
of Sections \ref{sec3}--\ref{sec4}.

\subsection{Notation}
Denote $\N=\{1,2,3,\ldots\}$,
$\Z_+=\{0,1,2,\ldots\}$, and $\R_+=[0,\iy)$.
For $x,y\in\R^2$, denote the standard inner product by $x\cdot y$
and the Euclidean norm by $\|x\|$.
Denote the boundary and the interior of a set $S\in\R^2$ by
$\pl S$ and $S^o$, respectively. Let
$\BB_r=\{x\in\R^2:\|x\|<r\}$, $\Sb_r=\calS\cap\pl\BB_r$, and
\[
\calS^n=(n^{-1/2}\Z_+)^2,\qquad n\in\N.
\]
If $V\in\R^2$ then $V_i$, $i\in\{1,2\}$ denote its components in the standard basis,
and vice versa: Given $V_i$, $i\in\{1,2\}$,
$V$ denotes the vector $(V_1,V_2)$. Both these conventions hold
also for random variables $V_i$ and processes $V_i(\cdot)$.
Members of $\R^2$ are considered as column vectors.

Denote by $C_b^2(\calS)$ the space of twice continuously differentiable functions defined on a neighborhood of $\calS$ that are bounded and have bounded first two derivatives.
For  $f \in C_b^2(\calS)$, denote by $\nabla f=(\nabla_i f)$
and $\nabla^2f=(\nabla^2_{ij}f)$
the gradient and Hessian.
For $\calS\subset\calO\subset\R^2$ and $f:\calO\to\R$, denote
\[
\pl^{n,i}_\pm f(x)=
f(x\pm n^{-1/2}e^{(i)})-f(x),\qquad x\in\calS^n,\ i\in\{1,2\},\ n\in\N,
\]
well-defined provided $x-n^{-1/2}e^{(i)}\in\calO$.

For $(\bX,d_\bX)$ a Polish space, let $C(\R_+,\bX)$ and $D(\R_+,\bX)$
denote the space of continuous and, respectively, \cadlag paths, endowed
with the topology of uniform convergence on compacts and, respectively,
the $J_1$ topology.

For $\xi\in D(\R_+,\R^N)$, $N=1,2$, an interval $I\subset\R_+$,
and $0\le \del\le S$, denote
\begin{align*}
\osc(\xi,I)&=\sup\{\|\xi(s)-\xi(t)\|:s,t\in I\},
\\
w_S(\xi,\del)&=\sup\{\|\xi(t)-\xi(s)\|:s,t\in[0,S],|s-t|\le\del\},
\\
\|\xi\|^*_S&=\sup\{\|\xi(t)\|:t\in[0,S]\},
\end{align*}
and by $|\xi|(t)$ the total variation of $\xi$ in $[0,t]$, taken with the Euclidean norm in $\R^N$.

$\iota:\R_+\to\R_+$ is the identity map.
For $k \in \Z$ and $x\in\R$ with $k\le x <k+1$, denote by $\sum_{j=1}^x$ 
the sum over $j\in\Z$ such that $1\le j\le k$.

For $x\in\R^2$, $b\in\R^2$ and $\Sig\in\frS$,
a $(b,\Sig)$-BM is a two-dimensional BM
with drift and diffusion coefficients $b$ and $\Sig$, respectively,
starting at $0$.
The Borel $\sigma$-field on a Polish space $\bX$ will be denoted as $\calB(\bX)$. Convergence in distribution is denoted by $\To$.
A sequence $\{P_n\}$ of probability measures on $D(\R_+,\bX)$ is said to be $C$-tight if it is tight with the usual $J_1$ topology on $D(\R_+,\bX)$ and every weak limit point $P$ is supported on $C(\R_+,\bX)$.
With a slight abuse of standard terminology,
a sequence of random elements (random variables or processes) is referred to
as tight when their probability laws form a tight sequence of probability 
measures, and a similar use is made for the term $C$-tight.
For a random variable $X$ with values in some Polish space $\bX$ and a probability measure $P$ on 
$(\bX, \calB(\bX))$, the notation $X \sim P$ denotes that the probability distribution of $X$ equals $P$.

\section{Main result and its applications}\label{sec2}

\subsection{Main result}
\label{sec21}
The processes $X^n$ and their rescaled versions $\hat X^n$ introduced
above are assumed to be defined on a probability space
$(\Om,\calF,\PP)$.
Henceforth,
$b$, $\Sig$ and $d^{(i)}$ are given by \eqref{095} and \eqref{eq:di},
and $b^*,d^{*,(i)},\theta^{(i)}_*$ and $\al_*$
by the corresponding quantities in \eqref{094}, \eqref{eq:thstar}
and \eqref{eq:alstar}.
As already mentioned, we shall focus on the range $\al_*\in(1,2)$,
a condition that can be rephrased
as $\tan\theta^{(1)}_*>\cot\theta^{(2)}_*$, or,
expressed directly in terms of the geometry of the model,
$\tan\theta^{(1)}>\cot\theta^{(2)}$.
In view of \eqref{eq:di}, the latter can also be written as
\begin{equation}\label{e2}
\nu_1\nu_2>\mu_1\mu_2,
\end{equation}
a condition assumed throughout what follows.

We recall the definition of the submartingale problem formulated in \cite{vw}.
It originates from an approach to defining
a measure on path space associated with a diffusion on a domain
with boundary, introduced in \cite{SV-B},
in which smooth domains were considered.
In the case of a wedge, \cite{vw} supplemented the formulation
of \cite{SV-B} by what is referred to below as the `corner property'.

The coordinate process on $C(\R_+,\calS)$ will be denoted by $Z$
throughout this paper and the filtration generated by this process will be referred to as the {\em canonical filtration}. Throughout,
a process defined on $(C(\R_+,\calS), \calB(C(\R_+,\calS))$
will be referred to as a martingale (resp., submartingale) if it is a martingale
(resp., submartingale) with respect to the canonical filtration.

\begin{definition}\label{def1}
{\bf (The submartingale problem)}
\\
(i) A tuple $\tilde\bD=(\tilde b,\tilde\Sig,\{\tilde d^{(i)}\})$ is {\it an admissible data}
if $\tilde b\in\R^2$, $\tilde\Sig\in\frS$,
and $\tilde d^{(i)}\in\R^2$, $\tilde d^{(i)}_i>0$
for $i=1,2$.
\\
(ii)
Let $z\in\calS$.
{\it A solution to the submartingale problem} for the admissible data
$\tilde\bD=(\tilde b,\tilde \Sig,\{\tilde d^{(i)}\})$ starting from $z$
is a probability measure $P_z$ on $C(\R_+,\calS)$
such that, with $E_z$ denoting the corresponding expectation,
the following statements hold.
\\
{\bf Initial condition.} One has
\[
P_z(Z(0)=z)=1.
\]
{\bf The submartingale property.}
Denoting
\begin{equation}\label{55}
\calL f=\frac{1}{2}{\rm trace}(\tilde\Sig\nabla^2f\tilde\Sig)+\tilde b\cdot\nabla f,
\quad \text{ and } \quad
M(t)=f(Z(t))-\int_0^t\calL f(Z(s))ds,
\end{equation}
$M(t)$ is a $P_z$-submartingale for any function $f\in C_b^2(\calS)$ that
is constant in a neighborhood of the origin and
satisfies $\tilde d^{(i)}\cdot\nabla f(x)\ge0$ for $x\in F_i$, $i=1,2$.
\\
{\bf The corner property.}
One has
\[
E_z\Big[\int_0^\iy \ONE_{\{0\}}(Z(t))dt\Big]=0.
\]
\end{definition}

The following wellposedness result for the above submartingale problem is from \cite{vw,lakner2023reflected}.
\begin{theorem}\label{th0}
\cite{vw,lakner2023reflected}
There exists a unique solution to the submartingale
problem for $\tilde\bD$ starting from $z\in\calS$ whenever
$\tilde\bD$ is an admissible data.
\end{theorem}

\proof
In \cite{vw}, a more general wedge was considered, and it was shown,
in the case $\tilde b=0$, $\tilde\Sig=\Id$, that there exists
a unique solution to the submartingale problem provided
that the parameter denoted there by $\al$
(defined as $\al_*$ in \eqref{eq:alstar}) is in the range $(-\iy,2)$.
The same is true for general $\tilde b$ and $\tilde\Sig=\Id$,
studied in \cite{lakner2023reflected}.
Now, in the case of the quarter plane, the parameter $\al$ of \cite{vw}
is always in the range $(-2,2)$ (since $\tilde d^{(i)}_i>0$ for both $i$), and consequently,
the assertion of the theorem holds in these cases, when $\tilde\Sig=\Id$.
In particular, for $\tilde b=0$, $\tilde\Sig=\Id$, this is stated in
\cite[Theorems 2.5 and 3.10]{vw}
and for general $\tilde b\in\R^2$ and $\tilde\Sig=\Id$,
this is \cite[Theorem 2.3]{lakner2023reflected}.

Next, for $\tilde\Sig\in\frS$, note that if $\ti X$ has the law $P_x$,
a solution with data $(\tilde b,\tilde\Sig,\{\tilde d^{(i)}\})$
starting from $x$, then
$\tilde\Sig^{-1}\ti X$ has law $P_y$, a solution with data
$(\tilde\Sig^{-1}\tilde b,\Id,\{\tilde\Sig^{-1}\tilde d^{(i)}\})$, starting from $y=\tilde\Sig^{-1}x$,
and vice versa.
Hence existence and uniqueness follow from the case $\tilde\Sig=\Id$.
\qed

Before stating the main result, we note that if $P_x$ is the solution
of the submartingale problem
then $X\sim P_x$ can be regarded as a process
with sample paths in $D(\R_+, \calS)$.

\begin{theorem}\label{th1}
Let $P_{\hat x}$ be the unique solution to the submartingale problem
with data $\bD=(b,\Sig,\{d^{(i)}\})$ starting from $\hat x$
and let $X\sim P_{\hat x}$. Then $\hat X^n\To X$ in $D(\R_+,\calS)$.
\end{theorem}

\newpage

\subsection{Queueing models}\label{sec22}

\subsubsection{GPS with parallelization slowdown}
In the GPS model introduced in \cite{parekh1993generalized},
a server divides its effort among several
streams of jobs according to fixed proportions.
Although it was introduced and analyzed with a general
number of streams, we describe the model here for the case of two streams,
which suffices for the purpose of relating it to our results.
Given positive proportions $\phi_i$, $i=1,2$
such that $\phi_1+\phi_2=1$, and  service rates $\mu^n_i$,
$i=1,2$, when jobs from both classes are present, the server's effort
is split according to the proportions $\phi_i$,
and the job at the head of the line in queue $i$
is served at rate $\phi_i\mu^n_i$, $i=1,2$.
When only the $i$-th queue is non-empty,
all effort is given to the job at the head of the line
in this queue, and thus it is served at rate $\mu^n_i$.

The GPS model with general service time distribution
and general number of streams
was considered in the heavy traffic limit in
\cite{ramanan2003fluid}, where convergence was proved
and the limit process was identified in terms of the extended Skorohod problem.
To relate the model to the setting studied in this paper, we note that
in the special case where the dimension is 2,
the limit process can alternatively be identified
as the unique solution to the submartingale
problem of \cite{vw}, with wedge angle $\pi/2$ and $\al=1$,
when the drift coefficient is $0$; and as the process constructed in \cite{lakner2023reflected} when the drift is nonzero.
An extension to a setting where some of the classes
may be strictly subcritical was studied in \cite{ramanan2008heavy}.
Although the proof of convergence
in \cite{ramanan2003fluid} was not
a mere application of the continuous mapping theorem,
the continuity of the extended Skorohod map was
instrumental to prove the convergence result.
Such a pathwise mapping is not available for the RBM constructed in \cite{vw}
when $\al>1$.

In practice, the implementation of GPS is often performed
by time slicing followed by a round robin scheduling.
This design necessitates preempt and resume operations,
and both may cause job switching overhead.
The effect of switching overhead in processor sharing and round robin
scheduling is widely acknowledged,
and has been described and analyzed e.g.\ in
\cite{gupta2008finding, thompson2009analysis, yun2009channel}.
(See also \cite{peng2022exact}
for job switching overhead analysis for priority scheduling.)
It is thus natural to posit that a fixed proportion of the processing
capacity is lost at times when parallelization takes place.

With this in mind, we consider here a variation of the model 
that will be called {\it GPS with parallelization slowdown}
(GPS--PS), which deviates from the GPS model in that
$\phi_1+\phi_2<1$ is assumed.
In this paper we only aim at the Markovian setting of the model.
Thus letting $\la^n_i$ denote the arrival rates
and $\mu^{\PS,n}_i$ the service rates, the queue-length process $X^n$
is a Markov process on $\calS^1$ with transition intensities
\[
r^{\PS,n}(x,y)=\begin{cases}
\la^n_i & \text{ if } y=x+e^{(i)},\,i=1,2,\\
\phi_i\mu^{\PS,n}_i & \text{ if } x\in\calS^o,\, y=x-e^{(i)},\,i=1,2,\\
\mu^{\PS,n}_i & \text{ if } x\in F_{i^\#},\, y=x-e^{(i)},\, i=1,2.
\end{cases}
\]
As before, $\la^n_i$ are assumed to satisfy \eqref{e7-}, while for $\mu^{\PS,n}_i$ we assume
\[
\hat\mu^{\PS,n}_i\doteq n^{-1/2}(\mu^{\PS,n}_i-n\mu^{\PS}_i)
\to\hat\mu^{\PS}_i.
\]
The critical load (or heavy traffic) condition is expressed by
\[
\la_i=\phi_i\mu^\PS_i,\qquad i=1,2.
\]
To relate this setting to the main result, take $\mu^n_i=\phi_i\mu^{\PS,n}_i$ and
$\nu^n_i=(1-\phi_i)\mu^{\PS,n}_i$. This gives
$(\mu_i,\hat\mu_i)=\phi_i(\mu^\PS_i,\hat\mu^\PS_i)$
and $\nu_i=(1-\phi_i)\mu^\PS_i$. Thus
\[
d^{\PS,(i)}=e^{(i)}-\frac{1-\phi_{i^\#}}{\phi_i}\frac{\mu^\PS_{i^\#}}{\mu^\PS_i}e^{(i^\#)}.
\]
Note that condition \eqref{e2} takes here the form $(1-\phi_1)(1-\phi_2)>\phi_1\phi_2$,
which is equivalent to our model assumption $\phi_1+\phi_2<1$.
As before we assume that $X^n(0) = x^n \in \calS^1$ and $\hat x^n = n^{-1/2}x^n \to \hat x \in \calS$.
As an application of Theorem \ref{th1} we obtain

\begin{corollary}
For the above GPS--PS model,
$\hat X^n\To X\sim P_{\hat x}$, the unique solution to
the submartingale problem with data $\bD^\PS=(b^\PS,\Sig^\PS,\{d^{\PS,(i)}\})$, starting 
from $\hat x$,
where
\[
b^\PS_i=\hat\la_i-\phi_i\hat\mu^\PS_i,
\qquad
(\sig^\PS_i)^2=\la_i+\phi_i\mu^\PS_i=2\phi_i\mu^\PS_i.
\]
\end{corollary}

\subsubsection{The coupled processor model}

Here, two servers work in parallel, each one serving a queue.
The rate of service offered to the job at the head of the line
of each queue depends on whether the other queue is empty or not.
A motivation for this model is a design in which
a server facing an empty queue is available to
help the other server; further motivation is given in \cite{boxma2013two}.
In Markovian setting, this model was analyzed in
\cite{fayolle1979two}; see also \cite[Section 9]{fayolle1999random}.
They characterized the steady state distribution via
a Riemann--Hilbert boundary value problem.
Later, \cite{cohen2000boundary} and \cite{boxma2013two}
considered extensions to
general service time distributions and, respectively,
Levy-driven queues, and analyzed the stationary joint workload process.

To relate the model to our notation, let $\mu^{\CP,n}_i$
and $\mu^{\CP,n}_i+\nu^{\CP,n}_i$, $i=1,2$ denote the service rates
to class $i$ when queue $i^\#$ is non-empty and, respectively, empty.
Thus $\nu^{\CP,n}_i$ can be regarded the service rate of the helping
activity. The  transition intensities of the queue-length proces $X^n$ are then
\[
r^{\CP,n}(x,y)=\begin{cases}
\la^n_i & \text{ if } y=x+e^{(i)},\,i=1,2,\\
\mu^{\CP,n}_i & \text{ if } x\in\calS^o,\, y=x-e^{(i)},\,i=1,2,\\
\mu^{\CP,n}_i+\nu^{\CP,n}_i & \text{ if } x\in F_{i^\#},\, y=x-e^{(i)},\, i=1,2.
\end{cases}
\]
This is precisely \eqref{e1} if one takes
$\mu^n_i=\mu^{\CP,n}_i$ and $\nu^n_i=\nu^{\CP,n}_i$,
and in full agreement with our main model once one
imposes all the assumptions on $\la^n_i$, $\mu^n_i$
and $\nu^n_i$ from Section  \ref{sec21}, namely \eqref{e7-}--\eqref{e2}, together with $n^{-1/2} x^n \to \hat x$.
Under these conditions we have

\begin{corollary}
For the above coupled processor model,
$\hat X^n\To X\sim P_{\hat x}$, the unique solution to
the submartingale problem with data $\bD^\CP=\bD$, starting from $\hat x$,
identical to that in Theorem \ref{th1}.
\end{corollary}

\subsection{Proof outline}\label{sec23}

Owing to the uniqueness stated in Theorem \ref{th0},
the main result can be established by showing that the sequence $\hat X^n$
is tight and that all its subsequential limits form solutions
to the submartingale problem.

Since a continuous map or oscillation inequalities are not available in
the setting under consideration, a different idea is required
to obtain tightness. The argument is based on showing that the following
property implies tightness: On sufficiently short intervals,
a trajectory of $\hat X^n$ can either lie in a small neighborhood of the origin,
or interact with at most one of the faces $F_i$, $i=1,2$.

The next main task is to prove that limits satisfy the corner property.
Our approach is as follows.
For  a constant $c>1$ and arbitrarily small $\eps>0$, we consider
a sequence of upcrossings from below $\eps$ to above $c\eps$,
and downcrossings from above $c\eps$ to below $\eps$,
of the process $\|\hat X^n(t)\|$.
The goal is to show that the expected cumulative time spent in upcrossings,
during any finite time interval $[0,S]$,
converges to zero upon sending $n\to\iy$ and then $\eps\to0$.
This can be achieved via upper estimates on
a single upcrossing duration and lower estimates on a single
downcrossing duration, which show that the former resides at a lower scale than the latter.
During a downcrossing contained in $[0,S]$,
the process is at least $\eps$ away from
the origin. Hence the number of times it switches
from visiting one of the faces $F_i$, $i=1,2$ to another,
is tight, as a sequence indexed by $n$.
The continuous mapping associated with reflection on each one of
the faces can therefore be used to argue that $\hat X^n$,
restricted to a single downcrossing, converges to such a restriction of an RBM.
In particular, downcrossing durations of the former are well approximated
by ones of the latter. Our RBM hitting time estimates are based
on a test function that was constructed in \cite{vw}
and played a central role there.
The downcrossing duration estimate obtained this way
scales like $\eps^{\al_*}$.

For upcrossings, the above approach is not applicable
due to the lack of an analogous representation of the process
via a continuous mapping when it is near the origin.
Instead, the argument is based on the construction of
another test function that is applied directly to the prelimit process.
It gives an estimate on the expected hitting time
from below $\eps$ to above $c\eps$ that is uniform in $n$.
This estimate scales like $\eps^2$.

These are then combined to produce a bound on the expected cumulative time
spent in $\BB_{\eps}$. The proof uses strong Markovity
of both the prelimit and the candidate limit process.
A crucial step is an argument
showing that the time spent by the prelimit process
in $F_1\cup F_2$ is controlled by the time spent at the origin.

All the above steps are carried out in the special case where
$\mu^n_i=n\mu_i=n\la_i=\la^n_i$,
under which the transition intensities of $\hat X^n$
are symmetric in the interior, and as a result, the candidate limiting RBM
has zero drift (but general diagonal diffusion matrix).
A use of Girsanov's theorem to transform symmetric intensities to more general
ones, followed by an estimate on the RN derivative,
gives a bound on the time spent in $\BB_{\eps}$
for the general nonsymmetric case.
The corner property then follows by Fatou's lemma.

To show the submartingale property, we start by deriving
the semimartingale representation
of $f(\hat X^n)$ for a test function $f$, followed by Taylor's expansion.
`Error' terms show up, involving the remainder terms in the expansion
and the time spent by the process in $\pl\calS=\{0\}\cup F_1\cup F_2$.
To show that these terms
converge to zero in probability, the aforementioned argument
that controls the time spent in $F_1\cup F_2$ by that at the origin
is used again. Combining it with the estimate on the time spent
in $\BB_{\eps}$, this allows us to control the time spent
in $\pl\calS$ and show that the error terms vanish in the limit.
The martingale terms in the prelimit representation, and the inequalities satisfied by the test function on the boundary, then give rise to
the submartingale property in the limit as $n\to\iy$.

Sections \ref{sec3}--\ref{sec4} are structured as follows.
Section \ref{sec3} contains preliminary steps in the proof:
In Section \ref{sec31}, certain elementary properties of Skorohod maps are recalled.
Section \ref{sec32} derives several equations satisfied by the processes $\hat X^n$
and Section \ref{sec33} proves their $C$-tightness.
Section \ref{sec34} provides the argument for controlling the time in $F_1\cup F_2$
in terms of the time at $0$.

The proof is then provided in Section \ref{sec4}, starting
with Section \ref{sec41} where estimates on individual
upcrossing and downcrossing durations
are derived. In Section \ref{sec42}, the aforementioned
sequence of upcrossings and downcrossings
is constructed and a bound is obtained on the expected cumulative time
spent in upcrossings. This bound is combined with
a Girsanov transformation in Section \ref{sec43} to deduce the corner property.
Finally, the submartingale property is proved in Section \ref{sec44}.

\section{Preliminaries}\label{sec3}

\subsection{Skorohod maps}\label{sec31}
We recall the definition and some elementary properties of
Skorohod maps on the half line and on the half plane.
The map that constrains a 1-dimensional trajectory to $\R_+$ is defined
as follows.
Given $\psi \in D(\R_+, \R)$ with $\psi(0)\ge 0$, there is a unique pair
$(\phi, \eta) \in D(\R_+, \R_+)^2$ such that $\phi = \psi + \eta$, $\eta(0)=0$, $\eta$ is nondecreasing and  $\int_{[0, \infty)} \phi(t) d\eta(t)=0$.
This pair is given by
\begin{equation}\label{36}
\eta(t)=\sup_{s\in[0,t]}\psi^-(s),\qquad \phi(t)=\psi(t)+\eta(t), \qquad t\ge0,
\end{equation}
The map from $D(\R_+,\R)$ to itself mapping $\psi\mapsto\phi$
is denoted by $\Gam_{\R_+}$.

Given $h\in\R^2$ with $h_1>0$, the map that constrains a 2-dimensional
trajectory to the half plane $\R_+\times\R$ in the oblique direction $h$
is defined as follows. Denoting $\bar h=h/\|h\|$,
given $\psi \in D(\R_+,\R^2)$, $\psi(0)\in\R_+\times \R$, there is a unique pair
$(\phi,\eta) \in D(\R_+,\R_+\times \R)^2$ such that $\phi=\psi+\eta$,
$\eta(0)=0$, $\eta(t)=\bar h|\eta|(t)$, and $\int_{[0,\iy)}\phi_1(t)d|\eta|(t)=0$.
This pair is explicitly expressed as
\[
\bar h_1|\eta|(t)=\eta_1(t)=\sup_{s\in[0,t]}\psi_1^-(s),\qquad t\ge0,
\]
and $\phi=\psi+\eta=\psi+\bar h|\eta|$.
The map from $D(\R_+,\R^2)$ to  $D(\R_+,\R_+ \times \R)$ mapping $\psi\mapsto\phi$ is denoted
$\Gam_{\R_+\times\R}^h$.
Analogously, for $h \in \R^2$ such that $h_2>0$, the map constraining a planar trajectory
to $\R\times\R_+$, in the direction $h$, is denoted by $\Gam_{\R\times\R_+}^h$.

With a slight abuse of notation,
we will use the same notation $\Gam_{\R_+}$, $\Gam_{\R_+\times\R}^h$
and $\Gam_{\R\times\R_+}^h$ for the corresponding maps
from $D([0,S],\R)$ to itself or $D([0,S],\R^2)$ to itself for finite $S$.

As follows from the explicit construction above, for $h$ in a compact subset $H$ of
$(0,\iy)\times\R$,
\begin{equation}\label{32}
\osc(\phi,[s,t])\le \kap\,\osc(\psi,[s,t]),
\end{equation}
whenever $\phi=\Gam_{\R_+\times\R}^h(\psi)$, where $\kap\in(0,\iy)$ is a constant that
does not depend on $\psi$, $s,t$ and $h\in H$.

\subsection{Equations for the rescaled process}\label{sec32}

Let $\calA_i$ and $\calD_i$, $i=1,2$
be independent Poisson point processes on $\R_+^2$ with Lebesgue intensity measure,
defined on $(\Om,\calF,\PP)$.
Denote by $\calA_i^c(ds,dz)=\calA_i^c(ds,dz)-ds\times dz$,
$\calD_i^c(ds,dz)=\calD_i^c(ds,dz)-ds\times dz$, $i=1,2$ their compensated versions.
The Markov process $X^n$ introduced in Section \ref{sec1} is equal
in distribution to the first component (again denoted as $X^n$)
of the tuple $(X^n,A^n,D^n)$ that constitutes the unique solution of the system
\begin{equation}
\label{eq:220}
\begin{split}
	 X^n(t) &= x^n+A^n(t)-D^n(t),
	 \\
	A^n_i(t) &= \int_{[0,t]\times \R_+} \ONE_{[0, \la^n_i]}(z) \calA_i(ds, dz),
	\\
	D^n_i(t) &= \int_{[0,t]\times \R_+} \left[\ONE_{[0, \mu^n_i]}(z) \ONE_{\calS^o}(X^n(s-)) + 
	\ONE_{[0, \mu^n_i+\nu^n_i]}(z) \ONE_{F_{i^\#}}(X^n(s-))\right] \calD_i(ds, dz).
\end{split}
\end{equation}
Define
\begin{align*}
T^n_i(t) &\doteq \int_0^t \ONE_{F_i}(X^n(s)) ds, \qquad i=1,2,
\qquad T^n_{12}\doteq T^n_1+T^n_2,
\\
T^n_0(t) &\doteq \int_0^t \ONE_{\{0\}}(X^n(s)) ds,
\qquad
 T^n_{\es}(t) \doteq \int_0^t \ONE_{\calS^o}(X^n(s)) ds.
\end{align*}
Note that
$D^n_i$ is a point process with compensator
\[
\mu^n_i T^n_{\es}(t) + (\mu^n_i +\nu^n_i) T^n_{i^\#}(t).
\]
In addition to the rescaled process
$\hat X^n_i(t) = n^{-1/2} X^n_i(t)$ already defined, let
\begin{equation}
\label{093}
 \hat A^n_i(t) \doteq n^{-1/2}[A^n_i(t) - \la^n_i t], \qquad \hat D^n_i(t) \doteq n^{-1/2}[D^n_i(t) -
 \mu^n_i T^n_{\es}(t) - (\mu^n_i +\nu^n_i) T^n_{i^\#}(t)].
\end{equation}
Then, for $t\ge 0$, we have
\begin{equation}\label{eq:hatxn}
	\hat X^n_i(t) = \hat x^n_i + \hat A^n_i(t) - \hat D^n_i(t)  + \hat \la^n_i t + \la_i n^{1/2} t - n^{-1/2}\mu^n_i T^n_{\es}(t) -n^{-1/2} (\mu^n_i +\nu^n_i) T^n_{i^\#}(t).
\end{equation}
Let $b^n_i=\hat\la^n_i-\hat\mu^n_i$ and note that $b^n\to b$.
Also, let
\[
d^{(i),n}=e^{(i)}-\frac{\nu^n_{i^\#}}{\mu^n_i}e^{(i^\#)},
\]
and note that $d^{(i),n}\to d^{(i)}$. Then by \eqref{eq:hatxn}
and the assumption $\la_i=\mu_i$,
\begin{align}\label{30}
\hat X^n(t)&=\hat Y^n(t)+\hat R^n(t)
\\
\label{30+}
\hat Y^n(t)&=\hat x^n+\hat A^n(t)-\hat D^n(t)+b^nt
\\
\label{31}
\hat R^n(t)&=\sum_in^{-1/2}\mu^n_id^{(i),n}T^n_i(t)
+\sum_in^{-1/2}\mu^n_ie^{(i)}T^n_0(t),
\end{align}
where, throughout, $\sum_{i=1}^2$ is abbreviated to $\sum_i$.
In this representation, $\hat R^n$ is a boundary term that constrains
 in direction $d^{(i),n}$ when $\hat X^n$ is on face $F_i$, $i=1,2$,
and in the direction $\sum_i\mu^n_ie^{(i)}$ when it is at the origin.

A variation of \eqref{30}--\eqref{31} is as follows.
Let
\begin{equation}\label{eq:circd}
\mathring{D}^n_i(t)
=n^{-1/2} \int_{[0,t]\times \R_+} \ONE_{[0, \mu^n_i]}(z) \calD_i^c(ds, dz).
\end{equation}
Letting $\mathring{M}^n=\mathring{D}^n-\hat D^n$, we have
\begin{align}
\mathring{M}^n_i(t)
&=
n^{-1/2}\int_{[0,t]\times \R_+} \left[\ONE_{[0, \mu^n_i]}(z)-\ONE_{[0, \mu^n_i]}(z) \ONE_{\calS^o}(X^n(s-)) - 
	\ONE_{[0, \mu^n_i+\nu^n_i]}(z) \ONE_{F_{i^\#}}(X^n(s-))\right] \calD^c_i(ds, dz)\nonumber
\\
&=
n^{-1/2}\int_{[0,t]\times \R_+} \left[\ONE_{[0, \mu^n_i]}(z)
\ONE_{\{0\}\cup F_i}(X^n(s-)) - 
	\ONE_{(\mu^n_i, \mu^n_i+\nu^n_i]}(z) \ONE_{F_{i^\#}}(X^n(s-))\right] \calD^c_i(ds, dz).\label{eq:1150}
\end{align}
Hence, with $\mathring{Y}^n=\hat Y^n+\hat D^n-\mathring{D}^n$,
we obtain from \eqref{30}--\eqref{30+},
\begin{align}\label{60}
\hat X^n(t)&=\mathring Y^n(t)+\mathring M^n(t)+\hat R^n(t)
\\
\label{60+}
\mathring Y^n(t)&=\hat x^n+\hat W^n(t)
=\hat x^n+\hat A^n(t)-\mathring D^n(t)+b^nt.
\end{align}
Equations \eqref{60}, \eqref{60+}, coupled with \eqref{31},
give an alternative to \eqref{30}--\eqref{31}, where now $\hat A^n$
and $\mathring D^n$ are mutually independent centered and scaled Poisson processes
and it is obvious that $\hat W^n$ converges to
a $(b,\Sig)$-BM.
Our proof uses both \eqref{30}--\eqref{31} and \eqref{60}--\eqref{60+}.

\subsection{$C$-tightness}\label{sec33}

\begin{proposition}\label{prop:tight}
			The sequence of processes $\{\hat X^n, n \in \N\}$ is $C$-tight in $D(\R_+, \R_+^2)$.
		\end{proposition}

\proof
We will use \eqref{30}--\eqref{31}. We first argue that $\{\hat Y^n\}$ is $C$-tight. Note that the sequence
$\hat A^n$ is $C$-tight as it converges to a BM.
The sequence $\hat D^n$ is $C$-tight as its $i$-th component is
a martingale with quadratic variation $[\hat D^n_i](t)
=n^{-1}D^n_i(t)$, which itself is a tight sequence of processes,
all limits of which are a.s.\ $c$-Lipschitz for a suitable constant $c$. Furthermore $b^n \to b$.
Combining these facts, $\{\hat Y^n\}$ is a $C$-tight sequence.

Now we turn to proving $C$-tightness of $\{\hat X^n\}$. Fix $S>0$. Because the initial conditions $\hat x^n$ converge,
it suffices to show that for any $\eps>0$ and $\eta>0$ there is $\del>0$ such that
\[
\limsup_n\PP(w_S(\hat X^n,\del)>\eps)<\eta.
\]
In the rest of this proof, $\del$ is always of the form $S/N$, some $N\in\N$.
For $\del>0$ and $k\in K_0\doteq\{1,2,\ldots,S/\del\}$, denote
$I_k=I^\del_k=[(k-1)\del,k\del]$. Then
\begin{equation}\label{50}
w_S(\hat X^n,\del)\le 2\max_{k\in K_0}\osc(\hat X^n,I_k).
\end{equation}
Consider the following partitioning of $K_0$:
\begin{align*}
K^n_1&=\{k\in K_0: \mbox{ for some } s,t\in I_k,\, s<t,\, \mbox{ there exist } i\in\{1,2\}, \mbox{ such that }
\\
&
\hspace{2em}
\hat X^n(s-)\in F_i, \hat X^n(t)\in F_{i^\#},
\hat X^n(u)\in\calS^o\cap \BB_\eps^c , \mbox{ for all } u\in[s,t)
\},
\\
K^n_2&=\{k\in K_0\setminus K^n_1:\min_{t\in I_k}\|\hat X^n(t)\|\le\eps\},
\\
K^n_3&=\{k\in K_0\setminus K^n_1:\min_{t\in I_k}\|\hat X^n(t)\|>\eps\}.
\end{align*}

Suppose that given $\eps>0$ and $\eta>0$ there is $\del>0$ such that
\begin{align}
\label{40}
&\limsup_n\PP(K^n_1\ne\emptyset)<\eta,
\\
\label{41}
&\limsup_n\PP(K^n_2\ne\emptyset \text{ and }
\max_{k\in K^n_2}\max_{t\in I_k}\|\hat X^n(t)\|>2\eps)<\eta,
\\
\label{42}
&
\limsup_n\PP(K^n_3\ne\emptyset \text{ and }
\max_{k\in K^n_3}\osc(\hat X^n,I_k)>\eps)<\eta.
\end{align}
Let us argue that, in this case,
\begin{equation}\label{420}
\limsup_n\PP(w_S(\hat X^n,\del)>6\eps)<3\eta.
\end{equation}
By \eqref{50}, on the event $w_S(\hat X^n,\del)>6\eps$,
there exists $k\in K_0$ for which $\osc(\hat X^n,I_k)>3\eps$.
If there exists such $k\in K^n_1$ then $K^n_1\ne\emptyset$,
an event having probability $<\eta$ from \eqref{40}.
If there exists such
$k\in K^n_3$ then in particular $\osc(\hat X^n,I_k)>\eps$,
an event having probability $<\eta$ from \eqref{42}. If there exists such
$k\in K^n_2$ then it is impossible that
\[
\max_{t\in I_k}\|\hat X^n(t)\|\le2\eps
\]
because the diameter of the set $\calS\cap \BB_{2\eps}$
is $2\sqrt{2}\eps<3\eps$. Hence, from \eqref{41}, this event also has probability $<\eta$,
showing \eqref{420}.
Note that \eqref{420}
gives $C$-tightness. It thus remains to prove \eqref{40}--\eqref{42}.

To show \eqref{40}, note that if $K^n_1$ is nonempty then for $s$ and $t$
as in the definition of $K^n_1$, one has (for $n$ sufficiently large) $\|\hat X^n(t)-\hat X^n(s)\|>\eps$.
Moreover, during the interval $[s,t)$, $\hat X^n$ takes values in $\calS^o$, hence, by
\eqref{31}, and recalling the definition of $T^n$, we have that $\hat R^n$ is flat over the interval. Hence by \eqref{30},
\[
\PP(K^n_1\ne\emptyset)\le\PP(w_S(\hat Y^n,\del)>\eps),
\]
and the claim follows by $C$-tightness of $\hat Y^n$.

Next, under the event indicated in \eqref{41},
there exist $k \in K_0$ and $s,t\in I_k$ satisfying $s<t$, such that
either $\|\hat X^n(s)\|\le\eps$, $\|\hat X^n(t)\|>2\eps$
or $\|\hat X^n(s)\|>2\eps$, $\|\hat X^n(t)\|\le\eps$.
Because the size of jumps is $n^{-1/2}$, it follows that there exist
two members of $I_k$, that are again denoted by $s<t$, such that
$\|\hat X^n(t)-\hat X^n(s)\|\ge \eps-n^{-1/2}$,
and during $[s,t]$, $\hat X^n$ does not visit $\BB_\eps$. In addition,
by the definition of $K^n_2$, $\hat X^n$ visits at most one of the faces $F_1$ or $F_2$
(in addition to $\calS^o$) during $[s,t]$.

Consider now the event in \eqref{42}.
Under this event,  there exists a $k \in K_0$, and 
$[s,t]\subset I_k$ with $\osc(\hat X^n,[s,t])>\eps$, during which at most
one of $F_1$ and $F_2$ is visited. 

Hence, in both cases, we can find a $k \in K_0$ and $s,t \in I_k$ with $s<t$, such that
 $\osc(\hat X^n,[s,t]) \ge \eps-n^{-1/2}$, and for all $u\in[s,t]$,
$\hat X^n(u)\in S^o\cup F_i$
for either $i=1$ or $2$. Suppose the former holds. Then
by \eqref{30}--\eqref{31}, for $u\in[s,t]$,
\begin{align*}
\hat X^n(u)-\hat X^n(s)
&=\hat Y^n(u)-\hat Y^n(s)+\hat R^n(u)-\hat R^n(s),
\\
\hat R^n(u)-\hat R^n(s)&=n^{-1/2}\mu^n_1d^{(1),n}(T^n_1(u)-T^n_1(s)).
\end{align*}
In particular, $\int_{[s, t)}\hat X^n_1(u)d|\hat R^n|(u)=0$.
If we let
\[
\check X^n(u)=\hat X^n(s+u), \qquad \check Y^n(u)= \hat X^n(s) + \hat Y^n(s+u)-\hat Y^n(s),
\qquad
0\le u\le t-s,
\]
we see that $\check X^n(u)=\Gam_{\R_+\times\R}^{d^{(1),n}}(\check Y^n)(u)$,
$u\in[0,t-s]$. Recalling that $d^{(1),n}$ converge,
it follows by \eqref{32} that, for some positive constant $c$, 
\[
\eps-n^{-1/2}\le\osc(\hat X^n,[s,t])\le c\,\osc(\hat Y^n,[s,t]).
\]
Clearly this conclusion also holds in the case where $F_2$ is visited (and $F_1$ is not).
This gives
\begin{align*}
&\PP(K^n_2\ne\emptyset \text{ and }\max_{k\in K^n_2}\max_{t\in I_k}\|\hat X^n(t)\|>2\eps)
\vee\PP(K^n_3\ne\emptyset \text{ and }\max_{k\in K^n_3}\osc(\hat X^n,I_k)>\eps)\\
&\quad\le
\PP(c\,w_S(\hat Y^n,\del)\ge\eps-n^{-1/2}).
\end{align*}
Now \eqref{41} and \eqref{42} follow on using once again the fact that $\{\hat Y^n\}$ are $C$-tight.
\qed

\subsection{A relation between the boundary processes}\label{sec34}

Let
\[
\hat\calF_t=\sig\{\calA_i([0,s]\times[0,z]),\calD_i([0,s]\times[0,z]),
i=1,2,\,s\in[0,t],\,z\in\R_+\}, \qquad t\in\R_+.
\]
Then, for all $n$, all the processes that were constructed in
Section \ref{sec32} are adapted to this filtration.

\begin{lemma}\label{lem:allthree}
Fix $S>0$. Let $s^n$ be a sequence of $\hat\calF_t$-stopping times,
and $t^n$ a sequence of random variables satisfying
$s^n\le t^n\le s^n+S$.
Then
there exist a constant $\kap_1\in(0,\iy)$ and a sequence of random variables $\xi^n\to0$
in probability such that
\[
T^n_{12}(t^n)-T^n_{12}(s^n)
\le \kap_1(T^n_0(t^n)-T^n_0(s^n))+\kap_1n^{-1/2}\|\hat X^n(s^n)\|+\xi^n.
\]
\end{lemma}
Note that the result does not require tightness of $s^n$.

\proof
Denote
\begin{equation}\label{35}
\tht^n_i = n^{-1}\mu^n_i+n^{-1}\nu^n_i,
\qquad
\beta^n_i = \mu^n_i  (\mu^n_i+\nu^n_i)^{-1}.
\end{equation}
By \eqref{e7}--\eqref{e7+} one has, as $n \to \infty$,
$$
\beta^n_i=\frac{\mu^n_i}{\mu^n_i+\nu^n_i} \to\beta_i\doteq\frac{\mu_i}{\mu_i+\nu_i}.
$$
Also, by \eqref{e2}, $\beta_1+\beta_2 <1$.
Hence there are $\eps>0$ and $n_0 \in \N$ such that for $n\ge n_0$,
$\beta^n_1+ \beta^n_2 <1-\eps$. In what follows, $n\ge n_0$.

Denote
$$
\bar H^n = n^{-1/2}\hat H^n,
$$
where $\hat H^n$ is each of the processes
$\hat A^n, \hat D^n, \hat X^n, \hat Y^n, \hat R^n$.
Denote $h^n=(1/\tht^n_1,1/\tht^n_2)$.
Then by \eqref{30}, $h^n\cdot\bar X^n=h^n\cdot\bar Y^n+h^n\cdot\bar R^n$.
Recall that $d^{(i),n}_i=1$ and $d^{(i),n}_{i^\#}=-\nu^n_{i^\#}/\mu^n_i$.
Hence by \eqref{31}, with $\bar\mu^n=n^{-1}\mu^n$
and $\bar\nu^n=n^{-1}\nu^n$,
\[
\bar R^n_i=\bar\mu^n_i T^n_i-\bar\nu^n_iT^n_{i^\#}+\bar\mu^n_iT^n_0,
\]
and
\begin{align*}
h^n\cdot\bar R^n
&=\frac{\bar\mu^n_1 T^n_1-\bar\nu^n_1T^n_2+\bar\mu^n_1T^n_0}{\bar\mu^n_1+\bar\nu^n_1}
+\frac{\bar\mu^n_2 T^n_2-\bar\nu^n_2T^n_1+\bar\mu^n_2T^n_0}{\bar\mu^n_2+\bar\nu^n_2}
\\
&=(\beta^n_1+\beta^n_2-1)T^n_{12}+(\beta^n_1+\beta^n_2)T^n_0.
\end{align*}
This gives
\[
(1-\beta^n_1-\beta^n_2)T^n_{12}(t)=(\beta^n_1+\beta^n_2)T^n_0(t)
-h^n\cdot\bar X^n(t)+h^n\cdot\bar Y^n(t).
\]
In view of the bound $(1- \beta^n_1-\beta^n_2) > \varepsilon$
and the nonnegativity of $h^n\cdot\bar X^n$,
\[
\eps(T^n_{12}(t^n)-T^n_{12}(s^n))
\le(\beta^n_1+\beta^n_2)(T^n_0(t^n)-T^n_0(s^n))
+h^n\cdot\bar X^n(s^n)+\xi^n_0,
\]
where
\[
\xi^n_0=\|h^n\cdot\bar Y^n(s^n+\cdot)-h^n\cdot\bar Y^n(s^n)\|^*_S.
\]
By the boundedness of $\beta^n_i$ and $h^n_i$,
the result will follow once it is shown that $\xi^n_0\to0$
in probability.

To this end,
note that $\bar D^n_i$ is an $\hat\calF_t$-martingale with quadratic variation
$[\bar D^n_i]=n^{-2}D^n_i$. Moreover, by \eqref{eq:220},
one has for some $C \in (0,\infty)$,
\[
D^n_i(s^n+S)-D^n_i(s^n)\le\calD_i([s^n,s^n+S]\times[0,Cn]).
\]
As a result, $\EE\{[\bar D^n_i](s^n+S)-[\bar D^n_i](s^n)\}\le CSn^{-1}$,
and by the Burkholder-Davis-Gundy inequality,
$\|\bar D^n_i(s^n+\cdot)-\bar D^n_i(s^n)\|^*_S\to0$ in probability.
A similar estimate holds for $\bar A^n_i$. Using this along with
the boundedness of $h^n$, $\hat x^n$ and $b^n$
in \eqref{30+} shows that $\xi^n_0\to0$ in probability.
\qed

\section{Proof of main result}\label{sec4}

Here we provide the proof of Theorem \ref{th1}, where the main two steps
are stated as parts (i) and (ii) of the following result. We will assume,
without loss of generality, that weak limit points $X$ of $\hat X^n$
are defined on the original probability space $(\Om,\calF,\PP)$.

\begin{proposition}
	\label{prop1}
	Let $X$ be a weak limit point of the sequence $\hat X^n$.
\\
(i) One has
	$$\E\Big[\int_0^\iy \ONE_{\{0\}}(X(t))dt\Big]=0.$$
(ii)
The process
\[
f(X(t))-\int_0^t\calL f(X(s))ds,
\qquad
\calL f\doteq\frac{1}{2}{\rm trace}(\Sig\nabla^2 f\Sig)+b\cdot\nabla f,
\]
with $f\in C_b^2(\calS)$ that is constant in a neighborhood of the origin and
satisfies $d^{(i)}\cdot\nabla f(x)\ge0$ for $x\in F_i$, $i=1,2$,
is a $\PP$-submartingale with respect to the filtration generated by $X$. Equivalently, if 
$P_{\hat x}$ is the law induced by $X$ on $C(\R_+,\calS)$ and
$M$ is given by \eqref{55} with data $\bD$  (recalling $Z$ is the coordinate process)
then $M$ is a $P_{\hat x}$-submartingale.
\end{proposition}

Parts (i) and (ii) of the above result are proved in Section \ref{sec43} and Section \ref{sec44},
respectively.

\noindent {\bf Proof of Theorem \ref{th1}.}
In view of the tightness of the sequence $\hat X^n$ stated in
Proposition~\ref{prop:tight} and the uniqueness
of solutions to the submartingale problem stated in
Theorem \ref{th0}, it suffices to show that
the law $P_{\hat x}$ of any weak limit point $X$ is a solution of the submartingale
problem with data $\bD$ starting from $\hat x$. To this end, note that
the initial condition follows from the assumed convergence
$\hat x^n\to\hat x$, whereas the corner property and the submartingale
property follow from parts (i) and (ii) of Proposition \ref{prop1},
respectively.
\qed

In the next two subsections we establish some estimates in preparation for the proof of Proposition~\ref{prop1}.

\subsection{Estimates on upcrossing and downcrossing durations}\label{sec41}

This subsection is concerned with upcrossings of
the modulus $\|\hat X^n(t)\|$ of the rescaled process
and downcrossings of the modulus
$\|X(t)\|$ of the candidate limit process, and develops
bounds on their durations. These bounds are proved
under a special choice of the rate parameters,
\begin{equation}\label{51}
\la^n_i=n\la_i,
\qquad
\mu^n_i=n\mu_i,
\qquad
\nu^n_i=n\nu_i,
\qquad
n\in\N,\, i=1,2.
\end{equation}

For $n\in\N$ and $x\in\calS^n$, let $\PP_x=\PP^n_x$ and $\EE_x=\E^n_x$ denote
the probability measure on $(\Om,\calF)$
for which $\hat X^n(0)=x$ a.s.\ and the corresponding expectation.

\begin{lemma}
	\label{lem:momsig}
Assume \eqref{51} holds.
For $\veps>0$ let $\st^n_{\veps} \doteq \inf\{t\ge 0: \|\hat X^n(t)\| \ge \veps\}$.
Then there is a constant $\kap \in (0,\infty)$, and for each $\veps>0$, an $n_0 \in \N$, such that, for all $n\ge n_0$,
$$\sup_{x \in \calS^n}  \E_x[\st^n_{\kap\veps} ] \le \veps^2.$$
\end{lemma}

\begin{proof}
Under assumption \eqref{51},
the generator of the process $\hat X^n$ is given by
\[
\calL^nf(x)
=\begin{cases}
\sum_in\la_i(\pl^{n,i}_+f(x)+\pl^{n,i}_-f(x)) & x\in\calS^n\cap\calS^o,\\
\sum_in\la_i\pl^{n,i}_+f(x)+n(\la_i+\nu_i)\pl^{n,i}_-f(x)
&x\in\calS^n\cap F_{i^\#},\, i=1,2,\\
\sum_in\la_i\pl^{n,i}_+f(x)
&x=0,
\end{cases}
\]
for $f:\calS^n\to\R$, where we recall the notation $\sum_i=\sum_{i=1}^2$.

	Fix $\veps>0$. Let $a_0 \doteq \veps (\la_1\wedge \la_2)^{-1/2}$
	and $A \doteq (\nu_1 \vee \nu_2) (\la_1\la_2)^{-1/2}$. Define $\Ps:\R_+^2  \to \R$ as
	$$\Ps(x) \doteq a_0^2 -(x_1^2\la_1^{-1} + x_2^2\la_2^{-1} + 2Ax_1x_2 (\la_1\la_2)^{-1/2}), \qquad x\in \R_+^2.$$
	For $x \in \R_+^2$, let $\|x\|_{\la} \doteq (x_1^2\la_1^{-1} + x_2^2\la_2^{-1})^{1/2}$.
	Using the inequality 
	$2x_1x_2 (\la_1\la_2)^{-1/2} \le x_1^2\la_1^{-1} + x_2^2\la_2^{-1}$ we see that
	$$\Ps(x) > 0 \quad \mbox{ when }\quad \|x\|_{\la} < a_0 (1+A)^{-1/2}.$$
	Now, for $x\in\calS^n\cap\calS^o$, we have
	$$\calL^n\Ps(x)=\sum_i n\la_i(\pl^{n,i}_+\Ps(x)+\pl^{n,i}_-\Ps(x)) = -4.$$
	For $x\in\calS^n\cap F_1$,
	\begin{align*}\calL^n\Ps(x)&=\sum_in\la_i\pl^{n,i}_+\Ps(x)
	+ n(\la_2+\nu_2)\pl^{n,2}_-\Ps(x)\\
	&= -n\Big(3n^{-1} + 2Ax_2 n^{-1/2}\frac{\la_1^{1/2}}{\la_2^{1/2}}\Big) -n\frac{\nu_2}{\la_2}(n^{-1} -2n^{-1/2}x_2)\\
	& \le -3,
	\end{align*}
	where the last inequality follows on recalling that $A \ge \nu_2(\la_1\la_2)^{-1/2}$.
	Similarly, for $x\in\calS^n\cap F_2$, $\calL^n\Ps(x)\le-3$.
	Finally,
	$$\calL^n\Ps(0)=\sum_in\la_i\pl^{n,i}_+\Ps(0) = -2.$$
	Fix $A_0 < (1+A)^{-1/2}$. Let $\gamma^n_{A_0 a_0} \doteq \inf\{t\ge 0: \|\hat X^n(t)\|_{\la} \ge A_0 a_0\}$.
	Then  for $x \in \calS^n$ with $\|x\|_{\la} \le A_0a_0$,
	$$\Ps(\hat X^n(t \wedge  \gamma^n_{A_0 a_0})) = \Ps(x) + \int_0^{t\wedge  \gamma^n_{A_0 a_0}} \calL^n\Ps(\hat X^n(s)) ds + M^n_t \le \Ps(x) -2(t \wedge \gamma^n_{A_0 a_0}) + M^n_t,$$
	where, under $\PP_x$, $M^n$ is a martingale starting at $0$.
	We can find $n_0 \in \NN$ such that for all $n\ge n_0$, and $x \in \calS^n$ with $\|x\|_{\la} \le A_0 a_0$, we have that
	$\|x+ n^{-1/2}e^{(i)}\|_{\la} < (1+A)^{-1/2}a_0$.
	Taking expectations, and noting that $\Ps(\hat X^n(t \wedge \gamma^n_{A_0 a_0}))>0$ when $n\ge n_0$,  we see that, for  $n\ge n_0$,
	$$\E_x(t \wedge  \gamma^n_{A_0 a_0}) \le \frac{1}{2} \Ps(x) \le \frac{1}{2}a_0^2.$$
	Sending $t \to \infty$ and noting that $ \gamma^n_{A_0 a_0} \ge   \st^n_{A_0\veps}$, we now see that, for $n \ge n_0$, 
	$ \E_x( \st^n_{A_0\veps}) \le \frac{1}{2}a_0^2 =\frac{1}{2}\veps^2(\la_1\wedge \la_2)^{-1} $. The result follows on taking $\kappa =A_0 (2(\la_1 \wedge \la_2))^{1/2}$.
	\qed
	\end{proof}

Under assumption \eqref{51}, the candidate limiting RBM has
zero drift. The next lemma is concerned with such a process.
The proof uses hitting time estimates from \cite{vw}, that were developed for a zero drift, unit variance RBM. To this end, it uses the diagonal transformation
$\Sig^{-1}$ to transform an RBM with data $(0,\Sig,\{d^{(i)}\})$
to one with data $(0,\Id,d^{*,(i)}\})$.
We recall from \eqref{eq:thstar}--\eqref{eq:alstar} that
the parameter $\al^*\in(1,2)$ was defined in terms of the latter.

\begin{lemma}\label{lem:r1}
For $x\in\calS$,
let $P_x$ denote the solution to the submartingale problem for
the admissible data $\bD=(0,\Sig,\{d^{(i)}\})$ starting from $x$,
$E_x$ the corresponding expectation.
Then there exist constants
 $\veps_0 >0$, $a>0$ and $c>1$, such that for all $\veps \in (0,\veps_0)$
\[
\sup_{x\in\Sb_{c\eps}} E_x[e^{-\tau^{\veps}}]\le 1-a\eps^{\al_*},
\]
where
$$
\tau^{\veps} = \inf\{t\ge 0: \|Z(t)\| \le \veps\}.
$$
\end{lemma}

\begin{proof}
We note that in \cite[Sec.\ 3.3]{vw} it is shown that, in the unit variance, zero drift
case, the solution to the submartingale problem
is a strong Markov process and it has the Feller property.
Since one can obtain a solution to the submartingale problem with data  $(0,\Sig,\{d^{(i)}\})$
from one with data $(0,\Id,d^{*,(i)}\})$ via a diagonal transformation, it follows that the former is a Feller strong Markov process as well.

Consider $c>1$ (the value of which is to be determined later in the proof),
let $\eps_0>0$ be such that $c\veps_0 <1$,
and $\veps \in (0, \veps_0)$. 
Fix $x \in \RR_+^2$ with $\|x\| = c\veps$. Denote
$\st=\inf\{t:\|Z(t)\|\ge 1\}$. In this proof we
suppress $\veps$ in the notation for $\tau^{\veps}$.
One has $\st\w\tau<\iy$ $P_x$-a.s., as follows from
\cite[eq.\ (2.21)]{vw} in the unit variance case, and consequently
in our case as well.
Let, for $t\ge 0$, $\calF_t \doteq \sigma\{Z(s): s \le t\}$.
Then, using strong Markovity,
\begin{align*}
E_x[e^{-\tau}] & \le E_x[\ONE_{\{\tau<\st\}}+e^{-\tau}\ONE_{\{\st<\tau\}}]
\\
&\le P_x(\tau<\st)+E_x[E_x[e^{-(\tau-\st)}\ONE_{\{\st<\tau\}}|\calF_\st]]
\\
&=P_x(\tau<\st)+E_x[\ONE_{\{\st<\tau\}}E_{Z(\st)}[e^{-\tau}]]
\\
&\le P_x(\tau<\st)+P_x(\st<\tau)\sup_{\|y\|=1}E_y[e^{-\tau}]
\\
&=1-P_x(\st<\tau)+P_x(\st<\tau)\kappa_1
\\
&=1-aP_x(\st<\tau),
\end{align*}
where $\kappa_1=\sup_{\|y\|=1}E_y[e^{-\tau}]$ and
$a=1-\kappa_1$.
To prove that $\kap_1<1$, argue by contradiction and assume $\kap_1=1$.
Then by the Feller property, there exists $y$, $\|y\|=1$
such that $\tau=0$ $P_y$-a.s.,
contradicting sample path continuity of $Z$. It follows that $\kap_1<1$
and $a>0$.
 
Next, let $Y \doteq \Sig^{-1}Z$ and note that its law
is a solution to the submartingale problem with data
$(0,\Id,\{d^{*,(i)}\})$ starting from $\Sig^{-1}x$.
Let $\Ph: \R_+^2 \to \R_+$ be the test function from \cite{vw}
defined as
\[
\Ph(y)=\|y\|^{\al_*}\cos(\al_*\theta-\theta^{(1)}_*),
\]
where $y= (\|y\|\cos \theta, \|y\| \sin \theta)$.
Then the proof given in \cite{vw} to equation \cite[(2.13)]{vw}, which regards
a different pair of stopping times, applies to $(\eta,\tau)$ thanks to the fact
that these stopping times on the filtration generated by $Z$
are also stopping times on the filtration generated by $Y$,
as the two filtrations are equal. It gives
\[
E_x[\Ph(Y(\st))\ONE_{\{\st<\tau\}}+\Ph(Y(\tau))\ONE_{\{\tau<\st\}}] = \Ph(\Sig^{-1}x).
\]
Note that, for $\theta\in[0,\pi/2]$,
\[
\cos(\al_*\theta-\theta_*^{(1)})\ge \kappa_2 \doteq\cos(|\theta^{(1)}_*|\vee|\theta^{(2)}_*|)>0.
\]
Therefore, with $\kappa_3 \doteq \kappa_2 (\sigma_1 \vee \sigma_2)^{-\alpha_*}$
\[
\kappa_3 \|x\|^{\alpha_*} \le \kappa_2\|\Sig^{-1} x\|^{\alpha_*} \le \Ph(\Sig^{-1}x) \le (\sigma_1 \wedge \sigma_2)^{-\alpha_*}(P_x(\st <\tau) + \veps^{\al_*}).
\]
Recalling that $\|x\| = c\veps$, we have
$$
P_x(\st <\tau)  \ge \kappa_3 c^{\alpha_*} \veps^{\alpha_*}(\sigma_1 \wedge \sigma_2)^{\alpha_*} - \veps^{\al_*}.
$$
Choose $c>1$ such that
$\kappa_3 c^{\alpha_*}(\sigma_1 \wedge \sigma_2)^{\alpha_*}\ge2$ to obtain
$P_x(\st <\tau) \ge \veps^{\al_*}$. The result follows.
\hfill \qed
\end{proof}

\subsection{A sequence of upcrossings and downcrossings}\label{sec42}

The goal of this subsection is to prove Proposition \ref{prop:zeropre} below,
which gives a bound on the expected cumulative time that $\{\hat X^n(t), t\in[0,S]\}$
spends in $\BB_\eps$ in the special case when \eqref{51} holds.
This is done by constructing successive upcrossings and downcrossings
of $\|\hat X^n\|$ and using the estimates from Section \ref{sec41}.
Lemma \ref{lem:momsig} is directly applicable to individual
upcrossings, whereas Lemma \ref{lem:r1} regards an RBM,
and transforming it to an estimate on downcrossing durations
of $\|\hat X^n\|$, as in Lemma \ref{lem11} below, involves a convergence argument.

\begin{proposition}
	\label{prop:zeropre}
	Assume \eqref{51} holds. Then for all $S<\infty$,
	$$\lim_{\veps \to 0}\limsup_{n\to \infty} \E_{\hat x^n}\int_0^{S}
	\one_{\BB_\eps}(\hat X^n(t)) dt =0.$$
\end{proposition}

Throughout what follows, fix $\eps_0>0$, $a>0$ and $c>1$ as in Lemma \ref{lem:r1}.
For $\veps \in (0, \veps_0)$, consider the following sequence of stopping times for each $n \ge n_0$. Namely,
$\gamma_{-1}^n=0$, and for $k \in \Z_+$,
\begin{equation}\label{eq:1143}
\begin{aligned}
	\gamma^n_{2k} &= \inf\{t \ge \gamma^n_{2k-1}: \|\hat X^n(t)\| \le  \veps\},\\
	\gamma^n_{2k+1} &= \inf\{t \ge \gamma^n_{2k}: \|\hat X^n(t)\| \ge c \veps\},
\end{aligned}
\end{equation}
suppressing here, and in the entire construction below, the dependence
on $\eps$.

\begin{lemma}\label{lem11}
Assume \eqref{51} holds.
Let $N\in\N$ and $S\in(0,\iy)$. Then
\begin{align*}
	\limsup_{n\to \infty} \PP\Big(\sum_{k=1}^N (\gamma^{n}_{2k} - \gamma^{n}_{2k-1}) \le S\Big)
\le \exp\{S-a N\eps^{\al_*}\}.
\end{align*}
\end{lemma}

In preparation to proving Lemma \ref{lem11}
we state and prove Lemma \ref{lem:ctymap} below. The proof of Lemma~\ref{lem11}
and Proposition \ref{prop:zeropre} appear afterwards.

Fix $0<c_0<1$ throughout what follows.
For any given $f \in D(\R_+, \R^2)$ with finitely many jumps in any compact time interval, satisfying
\[
\sup_{t\in(0,\iy)}\|f(t)-f(t-)\|\le c_0\eps/4
\quad \text{ and } \quad \|f(0)\| > c_0\veps,
\]
we can construct recursively a unique path $g \in D(\R_+, \R_+^2)$
reflected obliquely on $\pl\calS$, along the direction $d^{(i)}$ on face $F_i$, $i=1,2$,
and absorbed when first visiting $\calS\cap\BB_{c_0\eps}$.
Although quite standard, we provide a construction of such a path in Appendix 
\ref{sec:A}.
We will denote this reflected/absorbed path $g$ obtained from $f$ as $\La^{\veps}(f)$.
The following result gives a continuity property and a relation to the submartingale problem.
\begin{lemma}\label{lem:ctymap}
(i) Let $\{Y^n\}$ be a sequence of processes with sample paths in $D(\R_+, \R^2)$
such that for each $n$, a.s.,
$Y^n$ has finitely many jumps in any compact interval,
$\sup_{t\in(0,\iy)}\|Y^n(t)-Y^n(t-)\|\to0$ as $n\to\iy$, and $Y^n(0) \in \Sb_{c\eps}$;
in particular, $\La^\eps(Y^n)$ is well defined for all large $n$.
Suppose moreover that $Y^n \to \xi+ W$, a.s., uniformly on compacts, where
$\xi$ is an $\Sb_{c\eps}$-valued random variable and $W$ is a $(0, \Sig)$-BM
independent of $\xi$. Then
	$U^n \doteq \La^{\eps}(Y^n)$ converges a.s., uniformly on compacts, to $U \doteq \La^{\eps}(\xi+W)$.
\\
(ii) For any $x\in\calS$, $\|x\|>c_0\eps$, the process $U=\La^\eps(x+W)$
is equal in distribution to the process $Z(\cdot\w\tau^\eps)$
under the solution to the submartingale
problem with data $(0,\Sig,\{d^{(i)}\})$ starting at $x$, where
$\tau^\eps=\inf\{t:\|Z(t)\|\le c_0\eps\}$.
\end{lemma}
\begin{proof}
	i. Assume without loss that the above processes are defined
	on $(\Om, \calF, \PP)$.
	Let $\Om_0 \in \calF$ be the full measure set  on which
	all the a.s.\ properties in the statement of the lemma hold (for every $n$). Then there exist $n_0  \in \NN$, and for each 
	$\om \in \Om_0$   and  $n\ge n_0$, $k_n \in \Z_+\cup\{\iy\}$,
	$0\le\sig^n_1< \sig^n_2 <\cdots < \sig^n_{k_n-1} < \sig^n_{k_n}$
	(precisely defined in Appendix \ref{sec:A})
	such that $\sig^n_1$ is the first time $U^n$ hits one of the faces
	$F_i$, $i=1,2$, and for $2\le N \le k_{n}-1$,
	$\sig^n_N$ is the first time after $\sig^n_{N-1}$ when $U^n$ hits a face distinct from the one it hit at $\sig^n_{N-1}$,
and finally at $\sig^n_{k_n}$, it first hits $\BB_{c_0\eps}$,
after which it is absorbed ($k_n=\iy$ corresponds to never visiting $\BB_{c_0\eps}$).
	Denote the analogous sequence for $U$ by $0\le\sig_1< \sig_2 <\cdots < \sig_{k}$. All these quantities depend on $\om$ which is suppressed in the notation.
	We will use the standard fact that the hitting time of a closed set is a continuous function of the trajectory at $\PP$-a.e.\ path of  a nondegenerate two-dimensional RBM in the half space.
Fix $S_0\in(0,\iy)$.
	Then using this result, by a recursive argument, it follows that,
	for $\om$ in a full measure set $\Om_1 \subset \Om_0$, there is an $n_1 \ge n_0$, such that for all $n \ge n_1$,
$\sig^n_j\w S_0 \to \sig_j\w S_0$, and $\sup_{t \in [0, \sig_j\w S_0]} \|U^n(t)-U^n(t)\| \to 0$ for all $j \le k$. The result follows.

ii. This is immediate from the fact that $U$ is a $(0,\Sig)$-BM
that reflects on $\pl\calS$
in the direction $d^{(i)}$ on face $F_i$, $i=1,2$, and gets absorbed in $\BB_{c_0\eps}$.
\qed
\end{proof}

\noi{\bf Proof of Lemma \ref{lem11}.}
In this proof, a continuous time Markov process on $(n^{-1/2}\Z)^2$
with generator $\sum n\la_i(\pl^{n,i}_++\pl^{n,i}_-)$, starting at $0$,
will be called an {\it $n$-simple random walk} ($n$-SRW).
Due to the special choice of parameters in \eqref{51}, note that
$\la^n_i=\mu^n_i=\mu_in$ and $b^n=0$. Hence by the definitions
\eqref{093}, \eqref{eq:circd} and \eqref{60+}
of $\hat A^n$, $\mathring{D}^n$ and $\hat W^n$,
it is seen that $\hat W^n$ is an $n$-SRW.

The statement of the lemma is concerned with the lengths of the time intervals $[\gamma^{n}_{2k+1}, \gamma^{n}_{2k+2}]$, which are determined by the paths of $\hat X^n$. However, the proof will use a construction that, for each $k$, extends
the paths $\hat X^n|_{[\gamma^{n}_{2k+1}, \gamma^{n}_{2k+2}]}$
to $[\gamma^{n}_{2k+1}, \infty)$ in a way that depends
on $k$ and is distinct from $\hat X^n$.
For this extension we introduce $\{\hat W^{*,n}_k, k\in\Z_+\}$,
which is a collection of mutually independent $n$-SRWs,
independent of the processes $\calA_i,\calD_i$ that were
used for constructing $\hat X^n$ in \eqref{eq:220}.

To construct the extension
we go back to equations \eqref{60}--\eqref{60+}.
Under \eqref{51}, we also have $d^{(i),n} = d^{(i)}$.
Moreover, by \eqref{eq:1143}, for $t \in [\gamma^{n}_{2k+1}, \gamma^{n}_{2k+2}]$,
$T^n_0(t)=T^n_0(\gamma^{n}_{2k+1})$.
Let $\Th_k=\Th^n_{\eps,k}$ denote the shift operator acting
on $D(\R_+,\R^2)$ as
\[
\Th_k\zeta(t)=\zeta(\gamma^n_{2k+1}+t)-\zeta(\gamma^n_{2k+1}),
\qquad t\ge0.
\]
Then, denoting
$\hat\xi^n_k=\hat X^n(\gamma^n_{2k+1})$,
we obtain from \eqref{eq:circd}--\eqref{60+},
\[
\hat X^n(\gamma^n_{2k+1}+t)=\hat\xi^n_k+\Th_k\Big\{\hat W^n
+\mathring{M}^n+\sum_in^{1/2}\mu_id^{(i)}T^n_i\Big\}(t),
\qquad t\in[0,\gamma^n_{2k+2}-\gamma^n_{2k+1}].
\]
Since $c_0\eps<\eps$, this gives for all large $n$,
\begin{equation}
\label{090}
\hat X^n(\gamma^n_{2k+1}+t)
=\La^\eps(\hat\xi^n_k+\Th_k\hat W^n+\Th_k\mathring{M}^n)(t),
\qquad t\in[0,\gamma^n_{2k+2}-\gamma^n_{2k+1}].
\end{equation}
Moreover, since $\gamma^n_{2k+1}$ are stopping times
on the filtration generated by $(\hat X^n,\hat W^n)$,
and the latter forms a strong Markov process,
it follows that $\Th_k\hat W^n$ are also $n$-SRW for each $k$.
We now define, for each $k$, new processes that agree with
$\hat X^n(\gamma^n_{2k+1}+\cdot)$, $\Th_k\hat W^n$
and $\Th_k\mathring{M}^n$ on $[0,\gamma^n_{2k+2}-\gamma^n_{2k+1}]$,
thus, in particular, respect relation \eqref{090}, but may differ from them on
$(\gamma^n_{2k+2}-\gamma^n_{2k+1},\iy)$.
Namely, for $k \in \Z_+$, define
\begin{align*}
\hat W^{n}_{k}(t) 	\doteq 
\begin{cases}
	\Th_k\hat W^n(t), &\mbox{ for } t \in [0, \gamma^{n}_{2k+2}- \gamma^{n}_{2k+1}]\\
	\hat W^n(\gamma^{n}_{2k+2} ) - \hat W^n(\gamma^{n}_{2k+1} )+\hat W^{*,n}_k(t-\gamma^{n}_{2k+2}+\gamma^n_{2k+1})
	&\mbox{ for } t \in (\gamma^{n}_{2k+2}-\gamma^n_{2k+1},\iy).
\end{cases}	
\end{align*}
Noting that both $\Th_k\hat W^n$ and $\hat W^{n,*}_k$ are $n$-SRW
and $(\hat W^n,\gamma^n_{2k+2})$ are independent of
$\hat W^{n,*}_k$, it follows that $\hat W^{n}_k$ is an $n$-SRW
for each $k$. It is also clear from the construction that, for $k\in \Z_+$,
\begin{equation}\label{091}
	\hat W^{n}_k \mbox{ is independent of } (\{\hat W^{n}_j, j<k\}, \{\hat\xi^n_j, j \le k\}).
\end{equation}
If we now let
\begin{align}\label{092}
 \hat M^{n}_k(t) \doteq \mathring{M}^n((\gamma^{n}_{2k+1} +t) \wedge \gamma^{n}_{2k+2})- \mathring{M}^n(\gamma^{n}_{2k+1}),
\qquad t\ge0,
\end{align}
then by \eqref{090}, we have
\[
\hat X^n(\gamma^n_{2k+1}+t)=\hat U^n_k(t),
\qquad t\in[0,\gamma^n_{2k+2}-\gamma^n_{2k+1}],
\]
where
\begin{equation}\label{096}
\hat U^n_k(t)=\La^\eps(\hat\xi^n_k+\hat W^{n}_k+\hat M^n_k)(t),
\qquad t\in\R_+.
\end{equation}
Denoting $\tau^{n}_k \doteq \inf\{t\ge 0:  \|\hat U^{n}_k(t)\| \le \veps\}$,
we see that $\tau^{n}_k = \gamma^{n}_{2k+2}- \gamma^{n}_{2k+1}$.
Whereas, for $t\in[0,\gamma^n_{2k+2}-\gamma^n_{2k+1}]$, \eqref{096} expresses the same relation as \eqref{090},
the construction achieves in addition an independence structure
for $\hat W^n_k$, namely \eqref{091}, which $\Th_k\hat W^n$
do not possess.

Toward taking the $n$ limit, note that $\{\hat\xi^n_k, k \in \Z_+\}$ take values in a compact set and so this sequence is automatically tight. As $n$-SRW,
$\hat W^{n}_k$ converge in distribution, as $n\to\iy$,
to $W_k$ which is a
$(0,\Sig)$-BM. Furthermore, if along some convergent subsequence $\{(\xi_k, W_k), k \in \Z_+\}$ denotes the weak limit of $\{(\hat\xi^n_k,  \hat W^{n}_k ), k \in \Z_+\}$, then $\|\xi_k\| = c\veps$ a.s., for every $k$ and moreover,
the dependence structure \eqref{091} transfers to the limit,
namely, for $k\in \Z_+$,
\begin{equation}\label{eq:indprop}
	W_k \mbox{ is a $(0,\Sig)$-BM independent of } (\{W_j, j<k\}, \{\xi_j, j \le k\}).
\end{equation}

Next, using the fact that $T^n_0(\gamma^{n}_{2k+2})=T^n_0(\gamma^{n}_{2k+1})$ we apply Lemma \ref{lem:allthree}
with $s^n=s^n_k=\gamma^{n}_{2k+1}$ and $t^n=t^n_k=(s^n_k+S_0)\w\gamma^{n}_{2k+2}$,
for a fixed finite $S_0$.
It is clear from \eqref{eq:1143} that $n^{-1/2}\|\hat X^n(s^n_k)\|\to0$ in probability.
Hence Lemma \ref{lem:allthree} shows that
$T^n_{12}(t^n_k)-T^n_{12}(s^n_k)\to0$ in probability.
By \eqref{092} and \eqref{eq:1150},
$\EE\{[\hat M^{n}_k](S_0)\}\le C\EE\{T^n_{12}(t^n_k)-T^n_{12}(s^n_k)\}$ for some constant
$C$.
Noting that $T^n_{12}(t^n_k)-T^n_{12}(s^n_k)\le S_0$ shows
that $\EE\{[\hat M^{n}_k](S_0)\}\to0$, hence
by the Burkholder-Davis-Gundy inequality and the fact that $S_0$
is arbitrary, one has for $k\in\Z_+$,
$$\hat M^{n}_k \to 0 \mbox{ in probability, as } n\to \infty.$$
Combining these observations,
we have, along the above subsequence,
$$(\hat\xi^n_k, \hat W^{n}_k, k \in \Z_+)
\Rightarrow (\xi_k , W_k, k \in \Z_+)
\qquad
\text{in}\quad (\calS \times D(\R_+, \R^2))^{\Z_+}.
$$
By appealing to Skorohod's representation theorem we assume without loss of generality that the convergence is a.s. Then, letting $U_k \doteq \La^{\veps}(\xi_k+W_k)$ and using
Lemma \ref{lem:ctymap}.i we now see that, along the subsequence,
\begin{equation}\label{eq:weakcgce}
(\hat\xi^n_k, \hat W^{n}_k, \hat U^{n}_k, k \in \Z_+) \to (\xi_k, W_k, U_k, k \in \Z_+),\quad \text{a.s.}\end{equation}

Let, for $k\in \Z_+$,
\begin{align*}
	\tau_k \doteq \inf\{t\ge 0: \|U_k(t)\| \le \veps\}.
\end{align*}
Then, along the subsequence,
\begin{align*}
	\limsup_{n\to \infty} \PP\Big(\sum_{k=1}^N (\gamma^{n}_{2k} - \gamma^{n}_{2k-1}) \le S\Big)
	&= \limsup_{n\to \infty} \PP\Big(\sum_{k=1}^N \tau^{n}_{k-1}\le S\Big)\\
	&\le \PP\Big(\sum_{k=1}^N \tau_{k-1}\le S\Big),
\end{align*}
where the inequality follows from the  convergence in \eqref{eq:weakcgce}, the continuity of $U_k$,
and the lower semicontinuity, at any continuous trajectory in $D(\R_+, \R^2)$, of the hitting time of a closed set, regarded as a function on $D(\R_+, \R^2)$.
(In fact, as in the proof of Lemma \ref{lem:ctymap}, this hitting time is continuous at $U_k$ a.s., however that is not needed at this part of the proof.)

It remains to show that
$\PP\big(\sum_{k=1}^N \tau_{k-1}\le S\big)
\le e^{S-a N\eps^{\al_*}}$.
For $x\in\calS$, $\|x\|>c_0\eps$, denote
\[
g(x)=\EE[e^{-\tau(x,W)}]
\qquad \text{where}\qquad
\tau(x,W)=\inf\{t\ge0:\|\La^\eps(x+W)(t)\|\le\eps\},
\]
and $W$ is a $(0,\Sig)$-BM.
Let $\calG_k \doteq\sigma\{\xi_j, j \le k\} \vee \sigma\{W_j, j<k\}$.
Then, using \eqref{eq:indprop}, the fact that $\|\xi_k\| = c\veps$ a.s.,
and then Lemma \ref{lem:ctymap}.ii and Lemma \ref{lem:r1},
one has for all $\eps \in (0, \eps_0)$
\begin{equation}\label{eq:condit}
	\E[e^{-\tau_k} \mid \calG_k] =g(\xi_k)\le\sup_{x\in\Sb_{c\eps}}g(x)
	\le 1 - a \veps^{\alpha_*}
	\qquad \text{a.s.}
\end{equation}
Thus by successive conditioning,
\begin{align*}
\PP\Big(\sum_{k=1}^N\tau_{k-1}\le S\Big)
&\le e^S \E[e^{-\sum_{k=1}^N\tau_{k-1}}]
\\
&\le e^S(1-a\eps^{\al_*})^N
\\
&\le e^{S-a N\eps^{\al_*}},
\end{align*}
where the next to last inequality uses \eqref{eq:condit}.
This completes the proof of Lemma~\ref{lem11}.
\qed

\noi{\bf Proof of Proposition \ref{prop:zeropre}.}
The dependence of the stopping times $\{\gamma^n_j\}$ on $\eps$
will now be emphasized by writing $\gamma^{n,\eps}_j$.
	Let
	$$N_n^{\veps} = \max\{k \in \Z_+: \gamma_{2k}^{n,\veps} \le S\}.$$
	Let $\beta\in(\al_*,2)$.
	Then for all $\eps\in(0,\eps_0)$ and all large $n$,
\begin{align*}
\EE\int_0^{S} \one_{\{\|\hat X^n(t)\|\le \veps\}} dt
&\le \EE\Big[\one_{\{N_n^{\veps} \le\veps^{-\beta}\}}\sum_{k=0}^{N_n^{\veps}} (\gamma^{n,\veps}_{2k+1} - \gamma^{n,\veps}_{2k})\Big] +  S\PP(N_n^{\veps} > \veps^{-\beta})
\\
&\le \EE\Big[\sum_{k=0}^{  \veps^{-\beta}} (\gamma^{n,\veps}_{2k+1} - \gamma^{n,\veps}_{2k})\Big]
		+ S\PP\Big(\sum_{k=1}^{ \veps^{-\beta}} (\gamma^{n,\veps}_{2k} - \gamma^{n,\veps}_{2k-1}) \le S\Big).
\end{align*}
Let $\kappa$ be as in Lemma \ref{lem:momsig}.
Applying this lemma with $\eps$ replaced by $c\kappa^{-1}\eps$ gives
\[
\sup_{x\in\calS^n}\EE_x[\eta^n_{c\eps}]\le c^2\kappa^{-2}\eps^2.
\]
Hence
	\begin{align*}
		\EE\int_0^{S} \one_{\{\|\hat X^n(t)\|\le \veps\}} dt 
		\le c^2\kappa^{-2}\veps^{2-\beta} + S\PP\Big(\sum_{k=1}^{ \veps^{-\beta}} (\gamma^{n,\veps}_{2k} - \gamma^{n,\veps}_{2k-1}) \le S\Big).
	\end{align*}
Since the weakly convergent subsequence considered above \eqref{eq:indprop} was arbitrary,
we have from Lemma~\ref{lem11} that
\[
\limsup_{n\to\iy}\EE\int_0^{S} \one_{\{\|\hat X^n(t)\|\le \veps\}}dt
\le c^2\kappa^{-2}\veps^{2-\beta}+S\exp\{S-a\eps^{-(\beta-\al_*)}\}.
\]
The result follows on sending $\veps \to 0$.
\hfill \qed

\subsection{Proof of the corner property}\label{sec43}

\ 

\noi{\bf Proof of Proposition \ref{prop1}.i.}
The first step is to extend Proposition \ref{prop:zeropre} from
the special case \eqref{51} to
$\la^n, \mu^n, \nu^n$ as in the main result.
Fix $S \in (0,\infty)$.
Let $\PP^n$ denote the probability law on $D([0, S], \calS^n)$
induced by the Markov process
$\{\hat X^n(t): 0 \le t \le S\}$ starting at $\hat x^n\in\calS^n$, under $\PP$.
To distinguish the special case, let
 $\ti \PP^n$ denote the probability law on $D([0, S], \calS^n)$
 induced by $\{\hat X^n(s): 0 \le t \le S\}$ starting at $\hat x^n\in\calS^n$,
 in the special case \eqref{51}.
 Denote by $\EE^n$ and $\tilde\EE^n$ the corresponding expectations. Denote the coordinate process on $D([0,S], \calS^n)$ by $Z^n$.
 Let $L^n_{S} \doteq \frac{d\PP^n}{d\ti \PP^n}$. Then there exists
$K \in (0,\infty)$, such that for all $n$ and $x \in \calS^n$,
		\begin{equation}\label{eq:rndbd}
			\ti \EE^n[L^n_{S} ]^2 \le K.
			\end{equation} This fact is standard but for completeness we give a proof in Appendix \ref{sec:B}.
		By the Cauchy-Schwarz inequality, we have
		\begin{align*}
		\EE\int_0^S\one_{\BB_\eps}(\hat X^n(t))dt
		&=
		\E^n\int_0^{S} \one_{\BB_\eps}(Z^n(t))dt \\
		&\le(\ti \EE^n[L^n_{S} ]^2)^{1/2} \Big(\tilde\E^n\Big(\int_0^{S} \one_{\BB_\eps}(Z^n(t))dt\Big)^2\Big)^{1/2}\\
			&\le K^{1/2} \Big(\tilde\E^n\int_0^{S} \one_{\BB_\eps}(Z^n(t))dt\Big)^{1/2}.
		\end{align*}
Hence by Proposition \ref{prop:zeropre}, $\lim_{\veps \to 0}\kap(\eps)=0$,
where
\[
\kap(\eps)=\limsup_{n\to \infty} \E\int_0^{S} \one_{\BB_\eps}(\hat X^n(t))dt.
\]
Recall that $X$ denotes a weak limit point of $\hat X^n$. By appealing to the Skorohod representation theorem, we can assume that $\|\hat X^n - X\|^*_S \to 0$ a.s.\ as $n\to \infty$.
This convergence and Fatou's lemma imply
\begin{align*}
	\E\int_0^S \ONE_{\{0\}}(X(t))dt &\le 
	\E\int_0^S \liminf_{n\to \infty}\ONE_{\BB_\eps}(\hat X^n(t))dt\\
	&\le    \liminf_{n\to \infty}\E\int_0^{S}\ONE_{\BB_\eps}(\hat X^n(t)) dt\le\kap(\eps).
\end{align*}
Sending $\eps\to0$ shows that the left-hand side is $0$.
Sending  $S\to \infty$ and applying the monotone convergence theorem
proves the result. \hfill \qed

We record the following consequence of
the above proof and Lemma \ref{lem:allthree}.

\begin{corollary}\label{cor:allz}
For all $S<\infty$, $T^n_0(S) + T^n_1(S) + T^n_2(S) \to 0$ in probability as $n \to \infty$.
\end{corollary}
\proof
	For $S<\infty$ and $\eps>0$,
	$$\limsup_{n\to \infty} \E_{\hat x^n}T^n_0(S) = \limsup_{n\to \infty} \E_{\hat x^n}\int_0^{S}
	\one_{\{0\}}(\hat X^n(t)) dt \le\limsup_{n\to \infty} \E_{\hat x^n}\int_0^{S}
	\one_{\BB_\eps}(\hat X^n(t)) dt.$$
The right-hand side was denoted in the proof of Proposition \ref{prop1}.i
by $\kappa(\eps)$, and it was shown that $\kap(0+)=0$.
It follows that $T^n_0(S)\to0$ in probability.
The result now follows upon applying Lemma~\ref{lem:allthree}
with $0=s^n<t^n=S$.
\qed

\subsection{Proof of the submartingale property}\label{sec44}

The last step of the proof is to show that
limit points satisfy the submartingale property.

\noi{\bf Proof of Proposition \ref{prop1}.ii.}
Let $f \in C_b^2(\calS)$ be constant in a neighborhood of origin and
$d^{(i)}\cdot\nabla f(x) \ge 0$ for all $x \in F_i$, $i=1,2$.
Let, for $t\ge 0$, $\calF_t \doteq \sigma\{X(s): 0\le s \le t\}$. 
We need to show that
 $\{M(t)\}_{t\ge 0}$ defined as
	$$M(t) \doteq f(X(t)) - \int_0^t \calL f(X(s)) ds, \; t \ge 0,$$
	is an $\calF_t$-submartingale.

Fix a subsequence, denoted again as $\{n\}$, along which $\hat X^n\To X$. From \eqref{eq:220},
	it follows that
	\begin{align}
	f(\hat X^n(t)) &= f(\hat x^n) + \sum_i \int_{[0,t]\times \R_+}
	\pl^{n,i}_+f(\hat X^n(s-))
	\ONE_{[0, \la^n_i]}(z) \calA_i(ds, dz)\nonumber\\
	&\quad +  \sum_i  \int_{[0,t]\times \R_+} 
	\pl^{n,i}_-f(\hat X^n(s-))\Big[\ONE_{[0, \mu^n_i]}(z) \ONE_{\calS^o}(\hat X^n(s-)) \nonumber\\
	\label{eq:220f}&\hspace{15em}+ 
	\ONE_{[0, \mu^n_i+\nu^n_i]}(z) \ONE_{F_{i^\#}}(\hat X^n(s-)) \Big] \calD_i(ds, dz).
\end{align} 
	Fix $0 \le t_1 \le t_2<\infty$ and let $\psi: C(\RR_+, \calS)\to \RR$ be a nonnegative bounded continuous function.
	To prove the submartingale property, it suffices to show that
	\begin{equation}\label{eq:submartineq}
		E\left[\psi(X(\cdot\wedge t_1)) (M(t_2)-M(t_1))\right] \ge 0.
	\end{equation}
	Letting, for $i=1,2$,
	\begin{align*}
		M^{n,A}_i(t) &\doteq \int_{[0,t]\times \R_+} 
	\pl^{n,i}_+f(\hat X^n(s-))
	\ONE_{[0, \la^n_i]}(z) \calA^c_i(ds, dz) \\
	M^{n,D}_i(t) &\doteq \int_{[0,t]\times \R_+} 
	\pl^{n,i}_-f(\hat X^n(s-))\Big[\ONE_{[0, \mu^n_i]}(z)
	\ONE_{\calS^o}(\hat X^n(s-) \\
	&\hspace{14em} + 
	\ONE_{[0, \mu^n_i+\nu^n_i]}(z) \ONE_{F_{i^\#}}(\hat X^n(s-))\Big] \calD_i^c(ds, dz),
	\end{align*}
	we can write, for $n$ large enough,
	\begin{align*}
		&f(\hat X^n(t_2)) - f(\hat X^n(t_1)) - \sum_i [(M^{n,A}_i(t_2)- M^{n,A}_i(t_1))
		+ (M^{n,D}_i(t_2)- M^{n,D}_i(t_1))]\\
		&= \sum_i\int_{t_1}^{t_2} 
		\la^n_i\pl^{n,i}_+f(\hat X^n(s)) ds
		+\sum_i\int_{t_1}^{t_2}\pl^{n,i}_-f(\hat X^n(s)) \Big[\mu^n_i \ONE_{\calS^o}(\hat X^n(s))
		+(\mu^n_i+\nu^n_i) \ONE_{F_{i^\#}}(\hat X^n(s))\Big] ds\\ 
	&= \sum_i\int_{t_1}^{t_2} \Big[\la^n_i\pl^{n,i}_+f(\hat X^n(s))
	+\mu^n_i\pl^{n,i}_-f(\hat X^n(s))\Big] ds
	- \sum_i\int_{t_1}^{t_2}\pl^{n,i}_-f(\hat X^n(s))\mu^n_i \ONE_{F_i}(\hat X^n(s)) ds\\
	&\quad  +\sum_i\int_{t_1}^{t_2}\pl^{n,i}_-f(\hat X^n(s))\nu^n_i\ONE_{F_{i^\#}}(\hat X^n(s)) ds,
	\end{align*}
	where in proving the second equality, we have used the fact that $f$ is constant in a neighborhood of the origin and so 
	$\pl^{n,i}_-f(\hat X^n(s)) \ONE_{\{0\}}(\hat X^n(s)) =0$ for $n$ large enough.

	Next, using Taylor's approximation, write for $i=1,2$,
\begin{align*}
	\pl^{n,i}_+f(\hat X^n(s)) 
	&= n^{-1/2}\nabla_i f(\hat X^n(s)) + n^{-1}\frac{1}{2} \nabla_{ii}f(\hat X^n(s)) + G^{n,A}_i(s),\\ 
\pl^{n,i}_-f(\hat X^n(s)) 
	&= -n^{-1/2}\nabla_i f(\hat X^n(s)) + n^{-1}\frac{1}{2} \nabla_{ii}f(\hat X^n(s)) + G^{n,D}_i(s).
\end{align*}
Here $G^{n,A}_i(s)=(2n)^{-1}(\nabla_{ii}f(\hat Y^{(i),n}(s))-\nabla_{ii}f(\hat X^n(s))$
where $\hat Y^{(i),n}(s)$ is a point on the line segment joining
$\hat X^n(s)$ with $\hat X^n(s)+n^{-1/2}e^{(i)}$. In particular,
$\|\hat Y^{(i),n}(s)-\hat X^n(s)\|\le n^{-1/2}$. Similar relations holds for $G^{n,D}_i$.
By the tightness of $\hat X^n$ and the boundedness and continuity of $\nabla_{ii}f$,
\begin{equation}\label{eq:148}\sup_{t_1\le s \le t_2}\sum_i n[|G^{n,A}_i(s)|+ |G^{n,D}_i(s)|] \to 0\end{equation}
in $L^1$ as $n\to \infty$.
Now note that, for $i=1,2$, and $s \in [t_1, t_2]$,
\begin{align}
 &[-\pl^{n,i}_-f(\hat X^n(s))
	\mu^n_i + \pl^{n,i^\#}_-f(\hat X^n(s))\nu^n_{i^\#}]\nonumber\\
	&= 	[\hat \mu^n_i \nabla_if(\hat X^n(s)) - 
	\hat \nu^n_{i^\#} \nabla_{i^\#}f(\hat X^n(s))] + 
	n^{1/2} [ \mu_i \nabla_if(\hat X^n(s)) - 
	\nu_{i^\#} \nabla_{i^\#}f(\hat X^n(s))] +\tilde G^n_i(s),\label{eq:214}
\end{align}
where
$$\tilde G^n_i(s) = \bar\nu^n_{i^\#}\Big(\frac{1}{2} \nabla_{i^\#i^\#}f(\hat X^n(s)) +  nG^{n,D}_{i^\#}(s)\Big) -\bar\mu^n_i\Big(\frac{1}{2}  \nabla_{ii}f(\hat X^n(s)) + nG^{n,D}_i(s)\Big).$$
Since, $\bar \mu^n_i$, $\bar \nu^n_{i^\#}$ are bounded, $f \in C_b^2$, and \eqref{eq:148} holds, we have from Corollary \ref{cor:allz}
that
$$\int_{t_1}^{t_2} |\tilde G^n_i(s)| \ONE_{F_i}(\hat X^n(s)) ds \to 0 \mbox{ in $L^1$, as } n \to \infty.$$
Similarly, since $\hat \mu^n_i$ and $\hat \nu^n_{i^\#}$ are bounded
$$\int_{t_1}^{t_2} |\hat \mu^n_i \nabla_if(\hat X^n(s)) - 
	\hat \nu^n_{i^\#} \nabla_{i^\#}f(\hat X^n(s))|\ONE_{F_i}(\hat X^n(s)) ds \to 0 \mbox{ in $L^1$, as } n \to \infty.$$
Also, by the property $d^{(i)}\cdot\nabla f(x) \ge 0$ for all $x \in F_i$, we have that
$$[ \mu_i \nabla_if(\hat X^n(s)) - 
	\nu_{i^\#} \nabla_{i^\#}f(\hat X^n(s))]\ONE_{F_i}(\hat X^n(s)) \ge 0.$$
Using the above observations in \eqref{eq:214}, we see that	
\begin{align*}
	\sum_i\int_{t_1}^{t_2} \Big[-\pl^{n,i}_-f(\hat X^n(s))
	\mu^n_i + \pl^{n,i^\#}_-f(\hat X^n(s))\nu^n_{i^\#}\Big]
	\ONE_{F_i}(\hat X^n(s))ds 
	\ge   H^n_1,
\end{align*}
where $H^n_1\to 0$ in $L^1$.
	Next, using \eqref{eq:148},
	\begin{align*}
		&\sum_i [
		\la^n_i\pl^{n,i}_+f(\hat X^n(s))
	+ \mu^n_i\pl^{n,i}_-f(\hat X^n(s))]\\ 
	&= \sum_i [\hat \la^n_i - \hat \mu^n_i]\nabla_if(\hat X^n(s))
	+ \frac{1}{2}\sum_i n^{-1}[\la^n_i+\mu^n_i]\nabla_{ii}f(\hat X^n(s)) +\check G^n(s),
	\end{align*}
	where $\sup_{t_1\le s \le t_2} |\check G^n(s)|\to 0$ in $L^1$.
	Since $\hat \la^n_i - \hat \mu^n_i \to b_i$ and  $n^{-1}[\la^n_i+\mu^n_i] \to \la_i+\mu_i=2\la_i$, it then follows that
	\begin{align*}
		&\sum_i \int_{t_1}^{t_2}[
		\la^n_i\pl^{n,i}_+f(\hat X^n(s))
	+ \mu^n_i\pl^{n,i}_-f(\hat X^n(s))]ds 
	= \int_{t_1}^{t_2} \calL f(\hat X^n(s)) ds + H^n_2,
	\end{align*}
	where $H^n_2\to 0$ in $L^1$.
	Combining the above, and letting $\bar M^n(t) \doteq \sum_i (M^{n,A}_i(t)+
		M^{n,D}_i(t))$, we now see that
	\begin{align*}
		f(\hat X^n(t_2)) - f(\hat X^n(t_1)) - \int_{t_1}^{t_2} \calL f(\hat X^n(s)) ds \ge \bar M^n(t_2) - \bar M^n(t_1) + H^n,
	\end{align*}
	where $H^n \to 0$ in $L^1$.
	Thus, since $\psi$ is nonnegative,
	\begin{align*}
		&\EE\left[\psi(\hat X^n(\cdot\wedge t_1)) \left(f(\hat X^n(t_2)) - f(\hat X^n(t_1)) - \int_{t_1}^{t_2} \calL f(\hat X^n(s)) ds\right)\right]\\
		&\ge \EE\left[\psi(\hat X^n(\cdot\wedge t_1)) \left(\bar M^n(t_2) - \bar M^n(t_1) + H^n\right)\right]
		=\EE\left[\psi(\hat X^n(\cdot\wedge t_1))H^n\right],
	\end{align*}
	where in the last equality we have used the martingale property of $\bar M^n$.
	Now sending $n \to \infty$ and recalling the weak convergence of $\hat X^n$ to $X$ we have
	\begin{align*}
		&\EE\left[\psi(X(\cdot\wedge t_1)) \left(f(X(t_2)) - f(X(t_1)) - \int_{t_1}^{t_2} \calL f(X(s)) ds\right)\right]\\
		&= \lim_{n\to \infty} \EE\left[\psi(\hat X^n(\cdot\wedge t_1)) \left(f(\hat X^n(t_2)) - f(\hat X^n(t_1)) - \int_{t_1}^{t_2} \calL f(\hat X^n(s)) ds\right)\right] \ge 0.
	\end{align*}
	This proves \eqref{eq:submartineq} and completes the proof of the proposition. 
\hfill \qed

\appendix
\section{Construction of $\La^{\veps}$}
\label{sec:A}

Fix $f \in D(\R_+, \R^2)$ with $f(0) \in \calS$ and $|f(0)| > c_0\veps$,
satisfying, for all $t>0$,
$\|f(t)-f(t-)\| \le \tilde \veps \doteq c_0\veps/4$ and having locally finitely many jumps.
For notational simplicity, denote the map $\Gam^{d^{(1)}}_{\R_+\times \R}$ 
(resp.\ $\Gam^{d^{(2)}}_{\R\times \R_+}$) as $\Gam_1$ (resp.\ $\Gam_2$).
Also let $\tilde F_i \doteq \{x \in \calS^c: \inf_{y \in F_i}\|x-y\| \le \tilde \veps\}$.
Define a sequence $g_N\in D(\R_+,\R^2)$, $N\in\Z_+$ recursively as follows.
Let $\sig_0=0$ and
$$g_0(t) \doteq f(t), \; t \ge 0.$$
For $N\in\N$, if $\sig_{N-1}<\iy$, let
$$
\gamma_N \doteq \inf\{t\ge \sig_{N-1}: g_{N-1}(t) \in \calS^c\}, \;
\eta_N \doteq \inf\{t\ge \sig_{N-1}: |g_{N-1}(t)| \le c_0\veps\}, \; \sigma_N \doteq \gamma_N \wedge 
\eta_N.
$$
If $\gam_N<\iy$ and $\gam_N<\eta_N$, let $i_N$ be the index $i\in\{1,2\}$
for which $g_{N-1}(\gam_N)=g_{N-1}(\sig_N)\in\tilde F_i$.
If $\sig_{N-1}=\iy$ let $\sig_N=\iy$.
Let $g_N(t)=g_{N-1}(t)$ for $t \in [0, \sigma_N)$. If $\sig_N<\iy$,
for $t\in[0,\iy)$, let
\begin{align*}
	g_N(t) =
	\begin{cases}
		g_{N-1}(\sigma_N) & \mbox{ if } \eta_N \le \gamma_N\\ 
		\Gam_{i_N}(g_{N-1})(t) & \mbox{ if } \gamma_N < \eta_N.
	\end{cases}
\end{align*}
Note that one always has $i_{N}\ne i_{N-1}$,
and so the construction alternates between $\Gam_1$ 
(reflection at $F_1$) and $\Gam_2$ (reflection at $F_2$).

This construction produces a sequence $g_N$, $N\in\Z_+$ satisfying
$g_N(t)=g_{N-1}(t)$ for $t\in[0,\sig_N)$. If $\sig_N\to\iy$ then the pointwise
limit $g(t)\doteq\lim_Ng_N(t)$ exists for every $t$, and we set
$\La^\eps(f)=g$. In this case the trajectory never gets absorbed in
$\BB_{c_0\eps}$. If, on the other hand, $\sig_N$ remain bounded
then absorption must occur, namely
there must exist $N$ for which $\eta_N<\gam_N$. Let $K$ be the first
such $N$ and set $\La^\eps(f)=g_K$. This completes the construction.

\section{Proof of estimate \eqref{eq:rndbd}.}\label{sec:B}

We begin by observing that one can give the following distributionally equivalent construction of $X^n$.
Let $\{\calA^n_i, \calD^n_i, \calB^n_i, i=1,2\}$ be mutually independent Poisson processes with intensities $\la^n_i$, $\mu^n_i$ and $\mu^n_i + \nu^n_i$, $i=1,2$, respectively.
Define
\begin{equation}
\label{eq:220n}
\begin{split}
	 X^n(t) &= x^n+A^n(t)-D^n(t),
	 \\
	A^n_i(t) &= \calA^n_i(t),
	\\
	D^n_i(t) &= \int_{[0,t]}\ONE_{\calS^o}(X^n(s-))d\calD^n_i(s)
	 + 
	\int_{[0,t]}\ONE_{F_{i^\#}}(X^n(s-)) \calB^n_i(s).
\end{split}
\end{equation}
Then the process $X^n$ has the same law as the process $X^n$ introduced in \eqref{eq:220}.
When $\la^n_i, \mu^n_i, \nu^n_i$ are replaced by $n\la_i, n\mu_i, n\nu_i$, the corresponding processes above are denoted as $\ti\calA^n_i, \ti\calD^n_i, \ti\calB^n_i, \ti X^n$.
We can find $C \in (0, \infty)$ such that, for all $n \in \N$,
\begin{equation}\label{eq:553}\sum_i \left(\left|\frac{\la^n_i}{n\la_i}- 1\right| + \left|\frac{\mu^n_i}{n\mu_i}- 1\right| + \left|\frac{\nu^n_i+ \mu^n_i}{n\nu_i + n\mu_i}- 1\right|\right) \le \frac{C}{n^{1/2}}.\end{equation}
Fix $S <\infty$ and let $\Sig=D([0,S],\R)$.
Let $P_a^{n,i}$ denote the law of $\calA^n_i|_{[0,S]}$,
regarded as a probability measure on $(\Sig, \calB(\Sig))$. Similarly, denote the law of $\ti \calA^n_i|_{[0,S]}$ as $\ti P_a^{n,i}$.
Then by Girsanov's theorem for point processes,
$$\frac{dP_a^{n,i}}{d\ti P_a^{n,i}} = \exp \left\{\log\left(\frac{\la^n_i}{n\la_i}\right) \calA^n_i(S) + n\la_i S\left(1- \frac{\la^n_i}{n\la_i}\right)   \right\}.$$
Denoting the expectation under $\ti P_a^{n,i}$ as $\ti E_a^{n,i}$, it then follows from \eqref{eq:553} that
$$\ti E_a^{n,i} \left(\frac{dP_a^{n,i}}{d\ti P_a^{n,i}} \right)^2 \le  \exp\left\{ n\la_iS \left(1 - \frac{\la^n_i}{n\la_i}\right)^2    \right\}
\le e^{C^2\la_iS}.$$
Similar estimates hold when $(\calA^n_i, \ti \calA^n_i)$ is replaced by $(\calD^n_i, \ti \calD^n_i)$ and $(\calB^n_i, \ti \calB^n_i)$.
From the mutual independence of the Poisson processes it now follows that $L^n_S$
satisfies
$$\tilde \EE^n(L^n_S)^2 \le e^{C^2\sum_i (\la_i + 2\mu_i + \nu_i) S}.$$
This proves \eqref{eq:rndbd}. \qed

\noi{\bf Acknowledgment.}
Research of RA supported by ISF (grant 1035/20). Research of AB supported by the NSF (DMS-2134107, DMS-2152577).


\bibliographystyle{is-abbrv}

\bibliography{version-subm}

%
%

\end{document}